\documentclass{article}
\usepackage[paper=a4paper,left=30mm,right=30mm,top=25mm,bottom=25mm]{geometry}
\usepackage{graphicx}
\usepackage{subcaption}
\usepackage{lmodern}
\usepackage{comment}
\usepackage{dsfont}
\usepackage{pifont}
\usepackage{amsmath}
\usepackage{amsthm}
\usepackage{amssymb}
\usepackage{booktabs}
\usepackage{mathtools}
\usepackage{enumitem}
\usepackage{color}
\usepackage{csquotes}
\usepackage[numbers]{natbib}
\usepackage{hyperref}
  
\newtheorem{theorem}{Theorem}[section]
\newtheorem{optimizationproblem}[theorem]{Optimization Problem}
\newtheorem{lemma}[theorem]{Lemma}

\DeclareMathOperator{\acosh}{acosh}
\DeclareMathOperator{\clamp}{clamp}

\makeatletter
\newcommand\nopagebreakhere{\par\nobreak\@afterheading}
\makeatother

\DeclareMathOperator{\Var}{Var} 
  
\newcommand{\de}{\mathrm{\,d}}
\newcommand{\CC}{\mathcal{C}}  
\newcommand{\R}{\mathbb{R}}
\providecommand{\keywords}[1]{\textbf{Keywords } #1}
    
\newcommand{\NN}{\mathrm{N}}
\newcommand{\RR}{\mathcal{R}}

\title{\textbf{The exact region between Chatterjee's and Blest's rank correlations}}
\author{Marcus Rockel}
\date{\today}

\begin{document}

\maketitle
\begin{center}
	\small\textit{
		Department of Quantitative Finance,\\
		Institute for Economics, University of Freiburg,\\
		Rempartstr.\;16, 79098 Freiburg, Germany,\\
        \texttt{marcus.rockel@finance.uni-freiburg.de} \\[2mm]
	}
\end{center}
\begin{abstract}
Exact regions between rank correlations describe the set of all pairs of values that two dependence measures can attain simultaneously on the same copula, and thus yield sharp inequalities between them. 
In this paper, we determine the exact region between Chatterjee's rank correlation~$\xi$ and Blest's rank correlation~$\nu$ over the class of all bivariate copulas. 
Our approach is based on a constrained optimization problem whose solution is characterized by Karush--Kuhn--Tucker conditions. 
This leads to a novel extremal copula family that uniquely traces the boundary of the region. 
For this family, we derive closed-form expressions for both $\xi$ and $\nu$, which provide an explicit parametrization of the exact attainable region.
\end{abstract}
\vspace{0ex}

\section{Introduction}
\label{sec:xi-nu-region}

Understanding and quantifying dependence between random variables is a central theme in probability and statistics.
To this end, numerous dependence measures have been proposed, each designed to capture particular aspects of the relationship between variables.
When two such measures are considered simultaneously, a natural question is how their values are related.
A precise way to study this relationship is through exact regions between dependence measures.
These regions describe all pairs of values that the two measures can simultaneously attain when evaluated on the same copula, and thereby yield sharp inequalities between them.
Let $\CC$ denote the set of all bivariate \emph{copulas}, i.e., all distribution functions on $[0,1]^2$ with uniform marginals.
For two dependence measures
$\delta_1,\delta_2\colon\CC\rightarrow\R$,
the corresponding \emph{exact region} is defined as
\[
\RR_{\delta_1,\delta_2}
:=
\{(\delta_1(C),\delta_2(C)) : C\in\CC\}
\subseteq \R^2 .
\]
Determining this set yields sharp inequalities relating the two measures,
since its boundary characterizes the extremal dependence structures for
which the corresponding bounds are attained.
In 2017, \cite{schreyer2017exact} established the exact region between the classical measures of concordance Spearman's rho and Kendall's tau. 
Since then, attainable regions for several other dependence measures have been investigated, see, for instance, \cite{bukovvsek2022exact,bukovvsek2023exact}.

Traditional measures of concordance, such as Spearman's rho and Kendall's tau, primarily capture monotone association and treat the variables symmetrically. 
Recently, Chatterjee's rank correlation~$\xi$ has emerged as a popular alternative in the literature \cite{chatterjee2020,chatterjee2021}. Unlike classical measures, $\xi$ is an asymmetric coefficient specifically designed to quantify the strength of directed functional dependence. It takes values in $[0,1]$, where $\xi=0$ characterizes independence between $X$ and $Y$, and $\xi=1$ indicates that $Y$ is perfectly functionally dependent on $X$, i.e., $Y=f(X)$ for some Borel measurable function $f$.
For a general pair of random variables $(X,Y)$, Chatterjee's rank correlation can be written as
\begin{align}\label{eq:chatt}
\xi(Y|X)
  :=
  \frac{\int_{\R} \Var\left(P(Y \ge y \mid X)\right)\de P^Y(y)}
       {\int_{\R} \Var\left(\mathbf 1_{\{Y \ge y\}}\right)\de P^Y(y)} .
\end{align}
Intuitively, the numerator measures how strongly the conditional tail probabilities $P(Y\ge y\mid X)$ fluctuate with $X$, while the denominator captures the intrinsic variability of the tail events $\{Y\ge y\}$; their ratio therefore quantifies the fraction of the variation in $Y$ that can be explained by $X$.

In the case of continuous marginal distribution functions, Chatterjee's rank correlation does not depend on the marginals.
Throughout the paper, a (bivariate) \emph{copula} is a distribution function on $[0,1]^2$ with uniform marginals.
If $(X,Y)$ is a random vector with continuous marginal distribution functions $F_X$ and $F_Y$, then the joint distribution of $(U,V):=(F_X(X),F_Y(Y))$ is uniquely determined by a copula $C$, i.e., $(U,V)\sim C$.
In terms of the copula~$C$, Chatterjee's rank correlation can be expressed as
\begin{align}\label{eq:xi}
  \xi(C)
  :=
  6 \int_0^1\!\!\int_0^1 \bigl(\partial_1 C(t,v)\bigr)^2 \de t \de v
  - 2 .
\end{align}
In this form, the coefficient had previously been considered already in \cite{dette2013copula}.
The appeal of $\xi$ has sparked extensive research in recent years and in particular it has been compared against other established rank-based measures.
For instance, recent works have established the exact regions between $\xi$ and Spearman's rho \cite{ansari2025exact}, between $\xi$ and Spearman's footrule \cite{rockel2025exact}, as well as the exact regions between $\xi$ and other dependence measures within the specific class of lower semilinear copulas \cite{fuchs2025exact}.

\emph{Blest's rank correlation}~$\nu$ was introduced in \cite{blest2000rank} as a variant of Spearman's rho that emphasizes agreement at one end of the ranking scale.
Motivated by applications such as judged competitions, Blest proposed~$\nu$ as a refinement of Spearman's~rho that ``attaches more significance to the early ranking'' of items.
It is defined in terms of a copula~$C$ by
\begin{equation}\label{eq:nu}
  \nu(C)
  \;=\;
  24 \int_0^1 \!\!\int_0^1 (1-u)\,C(u,v)\de u\de v \;-\;2,
\end{equation}
a normalization for which $\nu(M)=1$, $\nu(W)=-1$, and $\nu(\Pi)=0$.
Trivially, $\nu$ is consistent with respect to the concordance order (i.e., $C\le D$ pointwise implies $\nu(C)\le \nu(D)$) and we say that $\nu$ belongs to the family of rank-based measures of \emph{weak concordance}, such as Spearman's rho and Kendall's tau.
Recently, the exact region between Spearman's~rho and Blest's~rank correlation within the class of extreme-value copulas was established in \cite{tschimpke2025exact}.
For every such copula~$C$, the pair $(\rho(C),\nu(C))$ satisfies the sharp inequalities
\[
  \frac{2\,\rho(C)\,\bigl(1+\rho(C)\bigr)}{3+\rho(C)}
  \;\le\;
  \nu(C)
  \;\le\;
  \frac{2\,\rho(C)\,\bigl(5-\rho(C)\bigr)}{9-\rho(C)}.
\]
The lower boundary is attained by copulas corresponding to triangular Pickands dependence functions (which characterize bivariate extreme-value copulas) that emphasize concordance in the lower tail, whereas the upper boundary arises from their mirror counterparts, which stress agreement in the upper tail.

In the present paper we determine the exact region between Chatterjee's rank
correlation $\xi$ and Blest's rank correlation $\nu$ over \emph{all} copulas, i.e., we characterize the set
\(
\RR_{\xi,\nu}
=
\{(\xi(C),\nu(C)) : C\in\CC\}.
\)
Determining this region yields sharp inequalities relating the two dependence measures.
The following theorem states the main result of the paper and is illustrated in Figure~\ref{fig:attainable_region_xi_nu}.

\begin{theorem}[Exact $(\xi,\nu)$-region in the $b$--parametrization]\label{thm:region}
Let $b>0$ and define $\Xi,\NN:(0,\infty)\to\R$ by
\begin{align}
\Xi(b)=&
\begin{cases}
\dfrac{8\,b^{2}(7-3b)}{105}, & 0<b\le 1,\\[8pt]
\dfrac{
 183\,\gamma
 -38\,b\,\gamma
 -88\,b^{2}\gamma
 +112\,b^{2}
 +48\,b^{3}\gamma
 -48\,b^{3}
 -\frac{105\acosh(\sqrt{b})}{b}
}{210}, & b>1,
\end{cases} \label{eq:Xi-formula}\\[6pt]
\NN(b)=&
\begin{cases}
\dfrac{4\,b(28-9b)}{105}, & 0<b\le 1,\\[8pt]
\dfrac{
 \frac{87\gamma}{b}
 +250\,\gamma
 -376\,b\,\gamma
 +448\,b
 +144\,b^{2}\gamma
 -144\,b^{2}
 -\frac{105\acosh(\sqrt{b})}{b^{2}}
}{420}, & b>1,
\end{cases} \label{eq:N-formula}
\end{align}
with
\(
\gamma:=\sqrt{\frac{b-1}{b}}
\)
for $b>1$.
Then the exact $(\xi,\nu)$-region is convex, closed, and satisfies
\begin{equation}\label{eq:region}
\RR_{\xi,\nu}
=\Bigl\{\,(\Xi(b),y)\in\R^2:\ -\NN(b)\le y\le \NN(b),\ b\in[0,\infty]\Bigr\},
\end{equation}
where we use the endpoint conventions $\Xi(0)=\NN(0)=0$ and $\Xi(\infty)=\NN(\infty)=1$.
Furthermore, the upper and lower curved boundary branches of $\RR_{\xi,\nu}$ are traced uniquely by the copula family $(C_b)_{b\in\R\setminus\{0\}}$ defined in~\eqref{eq:clamped} for $b>0$ and \eqref{eq:reflection} for $b<0$, while the remaining boundary is the vertical segment
\[
\{(1,y): -1\le y\le 1\}.
\]
Moreover, $C_1$ maximizes the difference $\nu-\xi$ over all bivariate copulas.
\end{theorem}

\begin{figure}[t!]
\centering
\includegraphics[width=0.7\textwidth]{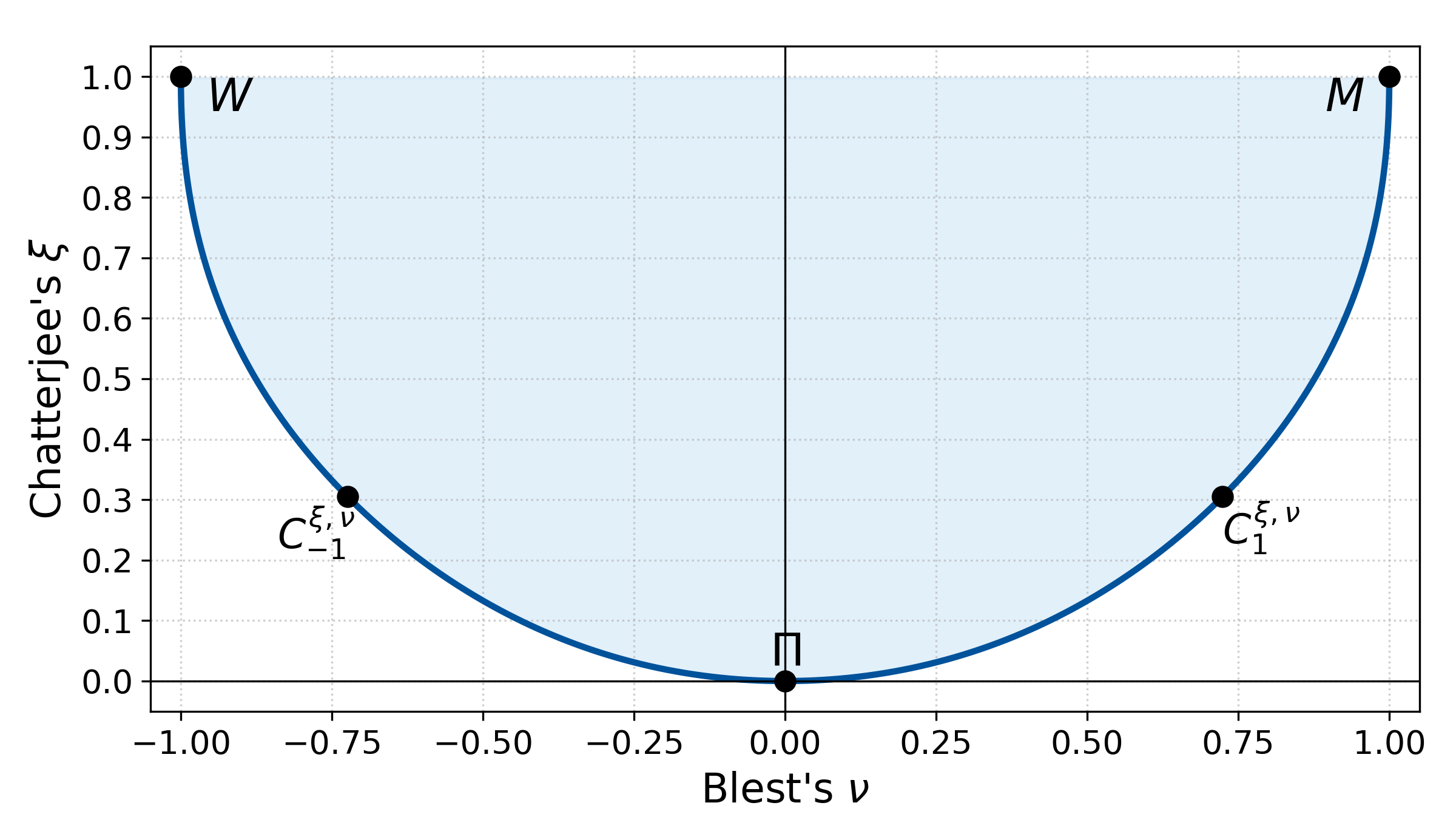}
\caption{
    The convex and closed attainable $(\xi,\nu)$-region.
    The boundary curve is uniquely traced by the copula family $(C^{\xi,\nu}_b)_{b\in\R\setminus\{0\}}$ introduced in Section~\ref{sec:novel-copula}.
    $\Pi(u,v):=uv$, $M(u,v):=\min\{u,v\}$ and $W(u,v):=\max\{u+v-1, 0\}$, for $u,v\in[0,1]$, denote the independence copula and the upper and lower Fréchet--Hoeffding copulas, respectively.
} 
\label{fig:attainable_region_xi_nu}
\end{figure}

The formulas in \eqref{eq:Xi-formula} and \eqref{eq:N-formula} are closely related.
Indeed, Lemma~\ref{lem:key-derivative} below shows that
\begin{align}\label{eq:key-derivative}
\NN'(b)=\frac{\Xi'(b)}{b}
\qquad (b>0).
\end{align}
While closed-form expressions for $\Xi$ and $\NN$ are available for all $b>0$,
no closed-form expression for the inverse of $\Xi$ appears to be available,
hence the parametrization used in~\eqref{eq:region}.

The rest of the paper is organized as follows.
Section~\ref{sec:novel-copula} constructs a novel copula family $(C_b)_{b\in\R\setminus\{0\}}$ and derives closed-form expressions for $\xi$ and $\nu$ in Theorem~\ref{thm:closed}.
Section~\ref{sec:optimization} formulates a constrained optimization problem that characterizes the upper boundary of the exact \((\xi,\nu)\)-region, and establishes that the copula family $(C_b)_{b>0}$ uniquely solves this problem, thus tracing the upper boundary of the exact \((\xi,\nu)\)-region.
All proofs are deferred to Section~\ref{sec:proofs}.

\section{\texorpdfstring
  {A novel copula family for the boundary of the $(\xi,\nu)$-region}
  {A novel copula family for the boundary of the (xi,nu)-region}}
\label{sec:novel-copula}


We begin by constructing a novel copula family $(C_b)_{b\in\R\setminus\{0\}}$ and then give explicit formulas for the corresponding values of Chatterjee's rank correlation and Blest's rank correlation. To define this family rigorously, we first note the following lemma.

\begin{lemma}
    \label{lem:unique_q}
Fix $b>0$ and, for $q\in[-1/b,1]$, define
\[
h_b^{(q)}(t)\;:=\;\clamp\!\Big(b\big((1-t)^2-q\big),\,0,\,1\Big),\qquad
\Phi(q)\;:=\;\int_0^1 h_b^{(q)}(t)\de t.
\]
Then $\Phi:[-1/b,1]\to[0,1]$ is continuous and strictly decreasing with $\Phi(-1/b)=1$ and $\Phi(1)=0$.
Consequently, for every $v\in[0,1]$ there exists a unique $q(v)\in[-1/b,1]$ such that
\begin{align}\label{eq:marginal}
\int_0^1 h_b^{(q(v))}(t)\de t\;=\;v.
\end{align}
Further, the map $(t,v)\mapsto h_b^{(q(v))}(t)$ is measurable on $[0,1]^2$.
\end{lemma}

With this lemma, fix now $b>0$ and let $q(v)\in[-1/b,1]$ enforce \eqref{eq:marginal}.
Define
\begin{equation}\label{eq:clamped_h}
h_b(t,v) = h^{\xi,\nu}_b(t,v):=\clamp\!\Big(b\big((1-t)^2-q(v)\big),\,0,\,1\Big),
\end{equation}
where $\clamp(y,\,\alpha,\,\beta) := \min\{\max\{y,\,\alpha\},\,\beta\}$ denotes the clamped ramp function with parameters $\alpha<\beta$.
Then, define the associated copula family $(C_b)_{b>0}$ by
\begin{equation}\label{eq:clamped}
C_b(u,v) = C^{\xi,\nu}_b(u,v):=\int_0^u h_b(t,v)\de t,
\end{equation}
so that $h_b(t,v)=\partial_1 C_b(t,v)$ is the first partial derivative of $C_b$.
The following holds.

\begin{lemma}\label{lem:cb_copula}
For every $b>0$, the function $C_b$ is a copula.
\end{lemma}

Like the partial derivative $h_b$, the family $(C_b)_{b>0}$ can also be written in explicit piecewise form.
However, the resulting expression is less transparent and does not provide additional insight into the exact $(\xi,\nu)$-region, so we omit it here.
A visualization of the conditional distributions $h_b(\cdot,v)$ for different values of $b$ is given in Figure~\ref{fig:cvxpy_xi_nu}.

\begin{figure}[ht]
    \centering
    \includegraphics[width=0.32\textwidth, trim=40 20 40 20, clip]{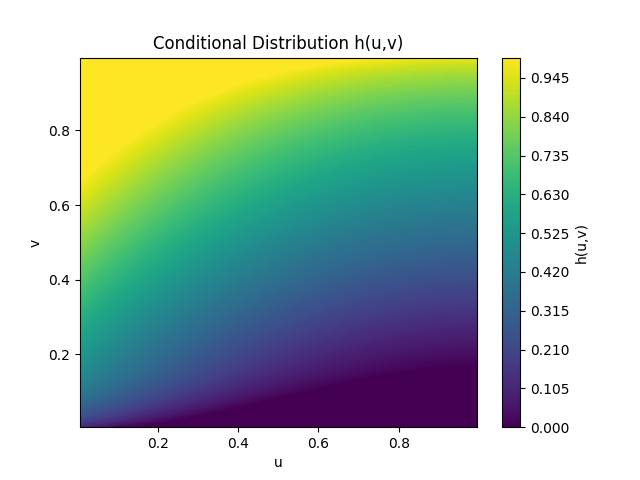}
    \includegraphics[width=0.32\textwidth, trim=40 20 40 20, clip]{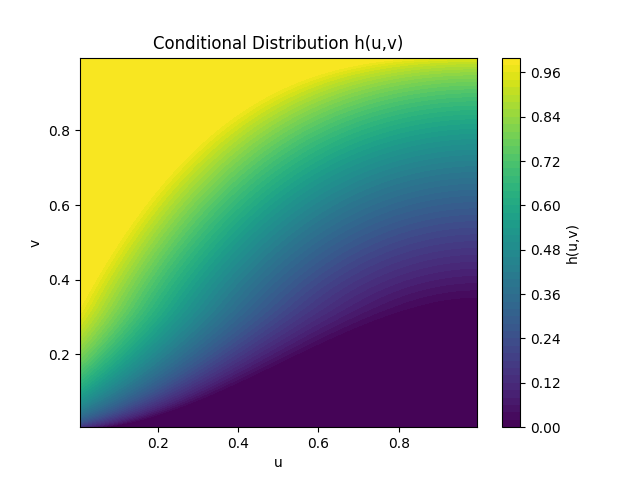}
    \includegraphics[width=0.32\textwidth, trim=40 20 40 20, clip]{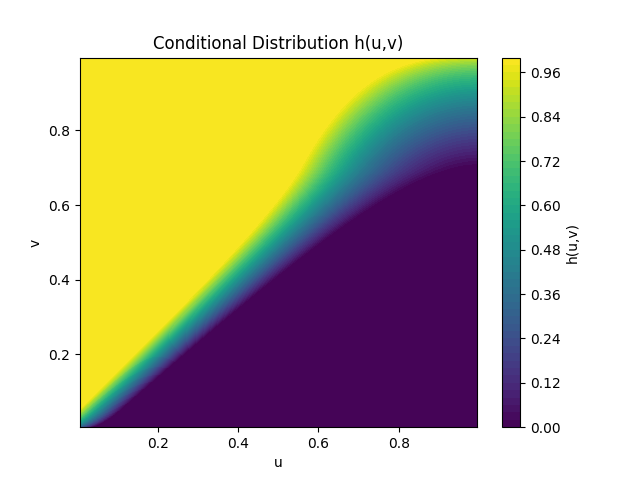}
    \caption{
        Conditional distribution plots for the copula $C_b$ defined in \eqref{eq:clamped} with $b=0.5$ (left), $b=1$ (middle), and $b=5$ (right).
        A large value of Blest's rank correlation for a fixed value of Chatterjee's rank correlation is achieved by allocating the functional dependence more heavily towards the top ranks (left side of each plot).
    }
    \label{fig:cvxpy_xi_nu}
\end{figure}

We will see in Section~\ref{sec:optimization} that the copula family $(C_b)_{b>0}$ defined in \eqref{eq:clamped} uniquely solves the optimization problem of maximizing $\nu$ at fixed $\xi$ in Theorem~\ref{thm:KKT}, and hence traces the upper boundary of the exact $(\xi,\nu)$-region from Theorem~\ref{thm:region}.
However, to explicitly state the boundaries of the exact $(\xi,\nu)$-region, we need closed-form expressions for $\xi(C_b)$ and $\nu(C_b)$, which we give here.

\begin{theorem}[Closed-form \(\xi\) and \(\nu\) for $C_b$]\label{thm:closed}
For \(b>0\), one has
\begin{align}\label{eq:xi_nu_formulas}
\xi(C_b) = \Xi(b),\qquad \nu(C_b) = \NN(b),
\end{align}
where $\Xi$ and $\NN$ are defined in \eqref{eq:Xi-formula} and \eqref{eq:N-formula}, respectively.
The two cases in \eqref{eq:Xi-formula} and \eqref{eq:N-formula} agree at \(b=1\), where \(\xi(C_1)=\frac{32}{105}\) and \(\nu(C_1)=\frac{76}{105}\).
\end{theorem}

We next note how Blest's rank correlation behaves under the reflections
\begin{align}\label{itm:concordance_reflections}
C^{\sigma_1}(u,v) &:= v - C(1-u,v),
\qquad C^{\sigma_2}(u,v) := u - C(u,1-v).
\end{align}
The first reflection $C^{\sigma_1}(u,v)$ corresponds to reversing the $X$-ranks, i.e.\ to the transformation $(X,Y)\mapsto(1-X,Y)$, but in general
\[
\nu(C^{\sigma_1})\ne-\nu(C),
\]
see \cite[Exa.~4]{genest2003blests}.
By contrast, for the second reflection
\[
C^{\sigma_2}(u,v)=u-C(u,1-v),
\]
which corresponds to reversing the $Y$-ranks, i.e.\ to the transformation $(X,Y)\mapsto(X,1-Y)$, it does hold that
\begin{equation}\label{eq:nu-sigma2}
\nu(C^{\sigma_2})=-\nu(C),
\end{equation}
see \cite{tschimpke2025exact}.
For our purposes, the lower-orthant ordering property
\[
D(u,v)\le E(u,v)\ \text{ for all }(u,v)\in[0,1]^2\implies \nu(D)\le \nu(E)
\]
and the reflection property \eqref{eq:nu-sigma2} are useful, as they allow us to deduce the lower boundary of the $(\xi,\nu)$-region from the upper one by reflection, as the transformation $Y\mapsto 1-Y$ preserves the variances in \eqref{eq:chatt} and hence $\xi$.
Indeed, one has
\[
\partial_1 C^{\sigma_2}(u,v) = 1 - \partial_1 C(u,1-v),
\]
and using a change of variables $v\mapsto 1-v$ and $\int_0^1\int_0^1 \partial_1 C(u,v)\de u\de v=1/2$, it follows from \eqref{eq:xi} that
\begin{align*}
\xi(C^{\sigma_2})
~ &= ~ 6 \int_0^1 \!\!\int_0^1 \bigl(1 - \partial_1 C(u,1-v)\bigr)^2 \de u \de v - 2 \\[4pt]
&= ~ 6 \Biggl( 1 - 2 \!\int_0^1 \!\!\int_0^1 \partial_1 C(u,v) \de u \de v
   + \!\int_0^1 \!\!\int_0^1 \bigl(\partial_1 C(u,v)\bigr)^2 \de u \de v \Biggr) - 2 \\[4pt]
&= ~ 6 \!\int_0^1 \!\!\int_0^1 \bigl(\partial_1 C(u,v)\bigr)^2 \de u \de v - 2
~=~\xi(C).
\end{align*}
Therefore, \eqref{eq:nu-sigma2} implies that the reflection $C\mapsto C^{\sigma_2}$ leaves $\xi$ invariant while flipping the sign of $\nu$.
Hence, the exact $(\xi,\nu)$-region is symmetric about the $\nu=0$ axis and the copula family $(C_b)_{b>0}$ is extended to trace the lower boundary of the $(\xi,\nu)$-region via the reflection $C\mapsto C^{\sigma_2}$.
More precisely, for $b<0$, define
\begin{align}\label{eq:reflection}
C_b(u,v):=u-C_{-b}(u,1-v),
\end{align}
so $C_b=C_{-b}^{\sigma_2}$.
With this extension, the notation \(C_b\) is now defined for all \(b\in\R\setminus\{0\}\), where \(b>0\) traces the upper boundary and \(b<0\) traces the lower boundary of the exact \((\xi,\nu)\)-region.
In particular, $C_b$ is a copula for every $b\ne 0$ and we obtain a copula family $(C_b)_{b\in\R\setminus\{0\}}$.
Therefore, to determine the exact \((\xi,\nu)\)-region, it suffices to determine the upper boundary, i.e.\ to maximize \(\nu\) at fixed \(\xi\); the lower boundary then follows by reflection.

\section{\texorpdfstring{Maximizing $\nu$ at fixed $\xi$}{Maximizing nu at fixed xi}}
\label{sec:optimization}

Our approach is based on expressing both rank correlations as functionals of the partial derivative $h_v(t) := \partial_1 C(t,v)$ of the copula.
Using this representation, the problem of determining the boundary of the $(\xi,\nu)$-region can be reformulated as an optimization problem over admissible families $(h_v)_{v\in[0,1]}$ satisfying the structural constraints that characterise copula derivatives.

Before turning to the optimization problem of maximizing $\nu$ at fixed $\xi$ in Theorem~\ref{thm:KKT}, we fix notation for the Banach-space optimization framework that will be used below. Our setup follows \cite[App.~A]{rockel2025exact}, which in turn is based on \cite[Ch.~3]{bonnans2013perturbation}. We briefly recall the main elements of this framework, which allows us to apply the method of Lagrange multipliers in infinite-dimensional spaces and to derive the necessary optimality conditions for our optimization problem.
More precisely, we consider optimization problems of the form
\begin{align}\label{eq:p}
  \min_{x\in X} f(x)
  \quad\text{subject to}\quad
  G(x)\in K,
\end{align}
where $X$ and $Y$ are Banach spaces, $f:X\to\R$ and $G:X\to Y$ are twice continuously Fr\'echet differentiable, and $K\subseteq Y$ is a non-empty closed convex set. We write
\[
  \mathcal{L}(x,a,\lambda)
  := a\,f(x)+\langle \lambda,G(x)\rangle,
  \qquad (x,a,\lambda)\in X\times\R\times Y^*,
\]
for the generalized Lagrangian and use the notation from \cite[App.~A]{rockel2025exact} throughout.

\begin{lemma}[KKT framework, \cite{bonnans2013perturbation,rockel2025exact}]\label{lem:KKT-framework}
Consider the optimization problem \eqref{eq:p} and let $x_0\in X$ be a feasible point, i.e.\ $G(x_0)\in K$. Then the following statements hold.

\begin{enumerate}[label=(\roman*)]
\item \label{itm:first-order-optimality}\emph{First-order optimality (KKT conditions).}
If $(a,\lambda)\in \R\times Y^*$ is a generalized Lagrange multiplier at $x_0$ with $a>0$, then after normalization $a=1$ the pair $(1,\lambda)$ satisfies
\[
D_x\mathcal{L}(x_0,1,\lambda)=0,
\qquad
\lambda\in N_K(G(x_0)),
\]
where
\(
N_K(y)
:=
\{\lambda\in Y^*:\langle \lambda,z-y\rangle\le0\ \text{for all } z\in K\}
\)
denotes the normal cone of $K$ at $y\in K$.

\item \label{itm:quadratic-growth}\emph{Quadratic growth.}
Let $\Lambda(x_0)$ denote the set of all multipliers $\lambda$ satisfying the KKT conditions above. 
If there exists $\beta>0$ such that
\[
\sup_{\lambda\in\Lambda(x_0)}
D^2_{xx}\mathcal{L}(x_0,1,\lambda)(k,k)
\ge \beta\|k\|^2
\qquad
\forall k\in X,
\]
then $x_0$ is a strict local minimizer of \eqref{eq:p}, i.e.\ there exist $C>0$ and a neighborhood $N_0$ of $x_0$ such that
\(
f(x)\ge f(x_0)+C\|x-x_0\|^2
\)
for all feasible $x\in N_0$.

\item \label{itm:convex-case}\emph{Convex case.}
If, in addition to the assumptions of \ref{itm:quadratic-growth}, $f$ and the feasible set are convex, then the quadratic growth condition implies that $x_0$ is the unique global minimizer of \eqref{eq:p}.
\end{enumerate}
\end{lemma}

The following characterization of bivariate copulas will be useful in the sequel, as it allows us to represent a copula via a family of functions describing its partial derivatives.

\begin{lemma}[A characterization of copulas, \cite{ansari2025exact}]
\label{charSIcop}
A function \(C\colon [0,1]^2 \to [0,1]\) is a bivariate copula if and only if there exists a family \((h_v)_{v\in [0,1]}\) of measurable functions \(h_v\colon [0,1]\to [0,1]\) such that
\begin{enumerate}[label=(\roman*)]
\item \label{charSIcop1} \(C(u,v) = \int_0^u h_v(t) \de t\) for all \((u,v)\in [0,1]^2\),
\item \label{charSIcop3} \(\int_0^1 h_v(t) \de t = v\) for all \(v\in [0,1]\),
\item \label{charSIcop2} \(h_v(t)\) is non-decreasing in \(v\) for all \(t\in [0,1]\).
\end{enumerate}
\end{lemma}

Lemma~\ref{charSIcop} allows us to represent a copula in terms of its partial derivative with respect to the first coordinate.
By \eqref{eq:xi}, Chatterjee's rank correlation is already expressed in terms of the first partial derivative of the copula.
To formulate the optimization problem in the same variables, we now rewrite Blest's rank correlation in terms of \(\partial_1 C\) as well.
Writing \(h(t,v):=\partial_1 C(t,v)\), one obtains by Fubini's theorem
\begin{align}\begin{aligned}\label{eq:intro-nu-forms}
  \nu(C)
  &= 24 \int_0^1\!\!\int_0^1 \int_{0}^{u} (1-u)\,h(t,v)\de t\de u\de v \;-\; 2  \\[0.3em]
  &= 24 \int_0^1\!\!\int_0^1 \int_{t}^{1} (1-u)\,h(t,v)\de u\de t\de v \;-\; 2
   = 12 \int_0^1\!\!\int_0^1 (1-t)^2\,h(t,v)\de t\de v \;-\; 2.
\end{aligned}\end{align}

Using Lemma~\ref{charSIcop} together with the representations \eqref{eq:xi} and \eqref{eq:intro-nu-forms}, the problem of maximizing $\nu$ at fixed $\xi$ can be expressed entirely in terms of the copula derivative $h$.
Lemma~\ref{charSIcop} is stated for measurable functions \(h_v\colon[0,1]\to[0,1]\), whereas Optimization Problem~\ref{opt:main} is posed over \(L^2((0,1)^2)\).
This is consistent, since any admissible copula derivative satisfies \(0\le h\le 1\) a.e., hence \(h\in L^\infty((0,1)^2)\subset L^2((0,1)^2)\).

\begin{optimizationproblem}[Maximize $\nu$ at fixed $\xi$]\label{opt:main}
Let $c\in(0,1)$.
Over $h\in L^2((0,1)^2)$, solve
\begin{align*}
\max &\quad \int_0^1\int_0^1 2(1-t)^2\,h(t,v)\de t\de v\\
\text{s.t.}&\quad 6\int_0^1\int_0^1 h^2(t,v)\de t\de v-2\le c, \qquad \int_0^1 h(t,v)\de t=v\ \text{ a.e.\ in }v,\qquad
0\le h\le 1.
\end{align*}
\end{optimizationproblem}

The objective functional appearing in the optimization problem has a particularly convenient structure.
Optimization Problem~\ref{opt:main} is a relaxed version of the original problem of maximizing $\nu$ at fixed $\xi$ over all copulas. 
Indeed, in the representation from Lemma~\ref{charSIcop}, a family $(h_v)_{v\in[0,1]}$ corresponds to a copula if and only if, in addition to the marginal condition
\(
\int_0^1 h(t,v)\de t = v
\)
for a.e.~$v\in(0,1)$ and the box constraint \(0\le h\le 1\), it also satisfies the monotonicity condition that
\[
v\mapsto h(t,v)
\]
is non-decreasing for almost every \(t\in[0,1]\). 
In Optimization Problem~\ref{opt:main}, we omit precisely this monotonicity constraint. 
This relaxation makes the variational problem substantially more tractable within the Banach-space KKT framework. 
As will be shown in Theorem~\ref{thm:KKT}, the relaxed problem nevertheless has the unique optimizer \(h_{b_c}\) from \eqref{eq:clamped_h}. 
This optimizer corresponds to \(C_{b_c}\) introduced in Section~\ref{sec:novel-copula} and, by Lemma~\ref{lem:cb_copula}, is indeed a copula.
Consequently, the optimizer of the relaxed problem is feasible for the original copula-constrained problem as well, and therefore also solves the problem of maximizing \(\nu\) over all copulas subject to the prescribed \(\xi\)-constraint.

The quantity $\nu(C)$ depends linearly on the copula derivative $h_v$, whereas $\xi(C)$ depends quadratically on $h_v$ through the integral of $h_v(t)^2$.
Consequently, the Lagrangian associated with the constrained optimization problem becomes a concave quadratic functional in the variable $h_v$.
Since the admissible set defined by the copula constraints is convex, any stationary point satisfying the KKT conditions is automatically the unique global optimum.
The optimization problem is solved in the following theorem, which establishes the previously introduced copula family $(C_b)_{b\in\R\setminus\{0\}}$ as the unique solution to Optimization Problem~\ref{opt:main} for appropriate choice of $b$ in dependence of $c$.

\begin{theorem}[Solution to Optimization Problem~\ref{opt:main}]\label{thm:KKT}
The optimizer of Optimization Problem~\ref{opt:main} is uniquely given by $h_{b_c}$ from \eqref{eq:clamped_h}, for some $b_c>0$ uniquely determined by $\xi(C_{b_c})=c$.
\end{theorem}

With Theorem~\ref{thm:KKT} at hand, it follows that the copula family $(C_b)_{b\in\R\setminus\{0\}}$ defined in \eqref{eq:clamped} uniquely solves the optimization problem of maximizing $\nu$ at fixed $\xi$.
Furthermore, as a consequence of these closed-form formulas from Theorem \ref{thm:closed}, we have now established the explicit parametric representation of the exact \((\xi,\nu)\)-region in Theorem~\ref{thm:region} above.

We conclude this section by numerically investigating the maximal gap \(\nu-\xi\) over several classical parametric copula families, see Table~\ref{tab:nu_minus_xi_max}.

\begin{table}[htbp]
  \centering
  \begin{tabular}{lrrrr}
    \toprule
    Family & Parameter & $\xi$ & $\nu$ & $\nu-\xi$ \\
    \midrule
    \((C_b)_{b\in\R\setminus\{0\}}\)        &      1 &        0.305 &   0.724 &     0.419 \\
    Clayton     &      1.925 &    0.324 &   0.719 &     0.395 \\
    Frank          &      5.746 &    0.312 &   0.695 &     0.383 \\
    Gaussian       &      0.682 &    0.284 &   0.665 &     0.381 \\
    Gumbel-Hougaard &      2.106 &    0.344 &   0.689 &     0.345 \\
    Joe            &      1.6 &      0.343 &   0.717 &     0.374 \\
    \bottomrule
  \end{tabular}
  \caption{
  Parameter values that approximately maximize the gap $\nu-\xi$ for the listed copula families, together with the 
  corresponding values of Chatterjee's rank correlation, Blest's rank correlation, and their difference.
  Except for \((C_b)_{b\in\R\setminus\{0\}}\), the entries are obtained by a dense grid search over the parameter, followed by cubic-spline interpolation of \(\nu\) and \(\xi\).
  }
  \label{tab:nu_minus_xi_max}
\end{table}

\section{Proofs}
\label{sec:proofs}

We first prove the auxiliary lemmas needed for the construction of the copula family and for the explicit computation of \(\xi(C_b)\) and \(\nu(C_b)\), culminating in the proof of Theorem~\ref{thm:closed}. We then establish the optimization result in Theorem~\ref{thm:KKT} and finally prove the main result of the paper, Theorem~\ref{thm:region}.

\subsection{Proofs for the copula construction and Theorem~\ref{thm:closed}}
\label{sec:proofs_novel_copula}

\begin{proof}[Proof of Lemma~\ref{lem:unique_q}]
For each $t\in[0,1]$, the unclamped value $b((1-t)^2-q)$ decreases strictly in $q$, hence
\[
h_b^{(q_1)}(t)\ge h_b^{(q_2)}(t)
\]
whenever $q_1<q_2$, with strict inequality on a set of positive measure.
Integrating over $t$ gives $\Phi(q_1)>\Phi(q_2)$, so $\Phi$ is strictly decreasing.
Since $0\le h_b^{(q)}\le1$ for all $q$ and $t\mapsto h_b^{(q)}(t)$ depends pointwise continuously on $q$ except at finitely many switching points (which have Lebesgue measure zero), dominated convergence implies continuity of $\Phi$.
At the endpoints,
\[
b((1-t)^2+1/b)\ge1
\]
for all $t$, giving $h_b^{(-1/b)}\equiv1$ and $\Phi(-1/b)=1$,
while
\[
b((1-t)^2-1)\le0
\]
for all $t$ gives $h_b^{(1)}\equiv0$ and $\Phi(1)=0$.
Hence, $\Phi$ is a continuous, strictly decreasing bijection from $[-1/b,1]$ onto $[0,1]$, and for each $v\in[0,1]$ there exists a unique $q(v)\in[-1/b,1]$ satisfying $\Phi(q(v))=v$.

Finally, to justify measurability of $(t,v)\mapsto h_b^{(q(v))}(t)$, define
\[
H(t,q):=\clamp\!\big(b((1-t)^2-q),\,0,\,1\big).
\]
The map $(t,q)\mapsto b((1-t)^2-q)$ is continuous, and $x\mapsto \clamp(x,0,1)=\min\{1,\max\{0,x\}\}$ is continuous, so $H$ is jointly continuous on $[0,1]\times[-1/b,1]$.
Since $\Phi$ is continuous and strictly monotone, its inverse $q(v)=\Phi^{-1}(v)$ is also continuous, and therefore the map $(t,v)\mapsto (t,q(v))$ is continuous.
The composition
\[
(t,v)\mapsto H\bigl(t,q(v)\bigr)=h_b^{(q(v))}(t)
\]
is thus continuous, in particular Borel measurable, on $[0,1]\times[0,1]$.
\end{proof}

\begin{proof}[Proof of Lemma~\ref{lem:cb_copula}]
By Lemma~\ref{lem:unique_q}, the map $(t,v)\mapsto h_b(t,v)$ is measurable and satisfies
\[
\int_0^1 h_b(t,v)\de t=v
\qquad\text{for all }v\in[0,1].
\]
It remains to check that, for each fixed $t\in[0,1]$, the map $v\mapsto h_b(t,v)$ is non-decreasing. Since $\Phi$ from Lemma~\ref{lem:unique_q} is continuous and strictly decreasing, its inverse $q(v)=\Phi^{-1}(v)$ is also continuous and strictly decreasing in $v$.
Note that for each fixed $t$, the map
\(
q\mapsto \clamp\!\Big(b\big((1-t)^2-q\big),\,0,\,1\Big)
\)
is non-increasing. Hence the composition
\[
v\longmapsto h_b(t,v)
=\clamp\!\Big(b\big((1-t)^2-q(v)\big),\,0,\,1\Big)
\]
is non-decreasing. Therefore all assumptions of Lemma~\ref{charSIcop} are satisfied, and $C_b$ is a copula for every $b>0$.
\end{proof}

The remainder of Section~\ref{sec:proofs_novel_copula} is essentially devoted to deriving the relations from \eqref{eq:xi_nu_formulas} in Theorem~\ref{thm:closed}.
Since both measures are given by double integrals, we split the derivation into the inner integral (Lemma~\ref{lem:1D}) and the outer integral (Lemma~\ref{lem:vq} and Theorem~\ref{thm:closed}).

Before we state the first lemma, observe that the optimizer $h_b^{(q)}(t,v)$ has a \emph{section-wise} structure: for each fixed $v$, the map $t\mapsto h_b(t,v)$ is a clamped decreasing function determined by a single parameter $q(v)$.
This observation allows us to turn the original two–dimensional integrals into one–dimensional ones in two steps: first, integrate along sections in \(t\); second, determine $v$ as a function of $q$ and integrate over $q$.
Introduce the change of variables
\begin{equation}\label{eq:unclamped_switches}
x := 1 - t, \qquad 
r := \sqrt{q} \ \ (\text{for }q \ge 0), \qquad 
R := \sqrt{q + \tfrac{1}{b}} \ \ (\text{for }q \ge -1/b).
\end{equation}
and define also
\begin{equation}\label{eq:section_switches}
X_a(q) := \min\{1, R\}, \quad
X_s(q) :=
\begin{cases}
0, & q < 0,\\
r, & q \ge 0,
\end{cases}
\quad
a(q) := \max\{0,\, 1 - R\}, \quad
s(q) := 1 - X_s(q).
\end{equation}
Thus, \(R\) and \(r\) describe where the unclamped parabola \(x \mapsto b(x^2 - q)\) would intersect the upper and lower levels $1$ and $0$, respectively, and \(X_a(q)\) and \(X_s(q)\) are the corresponding clipped values within the unit interval, while \(a(q)\) and \(s(q)\) transition $X_a(q)$ and $X_s(q)$ back to the \(t\)-domain.
Define
\begin{equation}\label{eq:primitives}
F(x;q):=\frac{x^5}{5}-\frac{2q}{3}x^3+q^2x,\quad
T(x;q):=\frac{x^3}{3}-q\,x\,,\quad
S(x;q):=\frac{x^5}{5}-\frac{q}{3}x^3,
\end{equation}
then we have the following lemma.

\begin{lemma}[Sections and one–dimensional forms]\label{lem:1D}
For each fixed $q$,
\begin{align}
\int_0^1 h_b(t,v)\de t
&= a(q)+b\big(T(X_a(q);q)-T(X_s(q);q)\big),\label{eq:marg-1D}\\
\int_0^1 h_b(t,v)^2\de t
&= a(q)+b^2\big(F(X_a(q);q)-F(X_s(q);q)\big),\label{eq:sq-1D}\\
\int_0^1 (1-t)^2 h_b(t,v)\de t
&= \frac{1-(1-a(q))^3}{3}+b\big(S(X_a(q);q)-S(X_s(q);q)\big).\label{eq:wgt-1D}
\end{align}
Furthermore, $q\mapsto v'(q)$ is well-defined and it holds
\begin{align}
\xi(C_b)&=6\!\!\int_{-1/b}^{1}\!\!\Big(a(q)+b^2\big(F(X_a(q);q)-F(X_s(q);q)\big)\Big)\,(-v'(q))\de q-2,\label{eq:xi-1D}\\
\nu(C_b)&=12\!\!\int_{-1/b}^{1}\!\!\Big(\frac{1-(1-a(q))^3}{3}+b\big(S(X_a(q);q)-S(X_s(q);q)\big)\Big)\,(-v'(q))\de q-2.\label{eq:nu-1D}
\end{align}
\end{lemma}

\begin{proof}[Proof of Lemma~\ref{lem:1D}]
Fix $q\in[-1/b,1]$ and $v$.
Let $s:=s(q)$ and $a:=a(q)$ be the points such that
\[
h_b(t,v)=
\begin{cases}
1,& t\in[0,a],\\[2pt]
b\big((1-t)^2-q\big),& t\in(a,s],\\[2pt]
0,& t\in(s,1].
\end{cases}
\]
For the marginal $\int_0^1 h_b(t,v)\de t$, one obtains
\(
\int_0^1 h_b(t,v)\de t
= a+b\int_a^s \big((1-t)^2-q\big)\de t.
\)
Set $x:=1-t$, so $\de t=-\de x$, $x(a)=1-a$, $x(s)=1-s$, and the integral becomes
\[
\int_a^s \big((1-t)^2-q\big)\de t
=\int_{x=1-a}^{1-s} (x^2-q)\,(-\de x)
=\int_{x=1-s}^{1-a} (x^2-q)\de x
=\Big(\frac{x^3}{3}-q\,x\Big)\Big|_{x=1-s}^{1-a}.
\]
With $R$ and $r$ as in \eqref{eq:unclamped_switches}, it is
\[
1-a=\min\{1,R\}=X_a(q),
\qquad 1-s=\begin{cases}r,&q\ge0,\\0,&q<0,\end{cases}=X_s(q).
\]
Hence
\(
\int_0^1 h_b(t,v)\de t
=a+b\big(T(X_a(q);q)-T(X_s(q);q)\big)
\)
with $T(x;q)$ as in \eqref{eq:primitives}, establishing \eqref{eq:marg-1D}.
Second, for the quadratic term in \eqref{eq:sq-1D}, on $(a,s]$ we have $h_b^2=b^2\big((1-t)^2-q\big)^2$, so
\begin{align*}
\int_0^1 h_b(t,v)^2\de t
= a+b^2\int_a^s \big((1-t)^2-q\big)^2\de t
= a+b^2\big(F(X_a(q);q)-F(X_s(q);q)\big),
\end{align*}
with $F(x;q)$ as in \eqref{eq:primitives}, so that \eqref{eq:sq-1D} is proved.
Third, for the weighted area in \eqref{eq:wgt-1D}, split the contribution of the plateau and the middle segment:
\[
\int_0^1 (1-t)^2 h_b(t,v)\de t
=\int_0^a (1-t)^2\de t + b\int_a^s (1-t)^2\big((1-t)^2-q\big)\de t.
\]
The first term equals $\int_{x=1-a}^{1} x^2\de x=\frac{1-(1-a)^3}{3}$.
The second equals
\[
b\int_{x=1-s}^{1-a} \big(x^4-qx^2\big)\de x
=b\Big(\frac{x^5}{5}-\frac{q}{3}x^3\Big)\Big|_{X_s(q)}^{X_a(q)}
=b\big(S(X_a(q);q)-S(X_s(q);q)\big),
\]
with $S(x;q)$ as in \eqref{eq:primitives}, proving \eqref{eq:wgt-1D}.

The one–dimensional representations \eqref{eq:xi-1D} and \eqref{eq:nu-1D} are obtained by substituting the section integrals \eqref{eq:marg-1D}--\eqref{eq:wgt-1D} into the defining two–dimensional formulas \eqref{eq:xi} and \eqref{eq:intro-nu-forms},
\begin{align}\label{eq:both_measures}
\xi(C_b)=6\!\!\int_0^1\!\!\int_0^1 h_b(t,v)^2\de t\de v -2,
\qquad
\nu(C_b)=12\!\!\int_0^1\!\!\int_0^1 (1-t)^2 h_b(t,v)\de t\de v -2.
\end{align}
By Lemma~\ref{lem:unique_q}, for each fixed \(b>0\) the function \(v\mapsto q(v)\) is defined implicitly by the marginal condition \(\int_0^1 h_b(t,v)\de t=v\), and the mapping \(v\mapsto q(v)\) is continuous and strictly decreasing from \(q=-1/b\) (when \(v=1\)) to \(q=1\) (when \(v=0\)).
Hence, it can be inverted to a mapping 
\(
q \mapsto v(q) = \int_0^1 h_b(t,q)\de t.
\)
For each fixed \(t\), the function \(q \mapsto h_b(t,q)\) is piecewise polynomial in \(q\) with breakpoints 
at \(a(q)\) and \(s(q)\), where
\[
a(q) = 1 - \sqrt{q + 1/b}, 
\qquad 
s(q) = 1 - \sqrt{\max\{q,0\}}.
\]
Since \(a(q)\) and \(s(q)\) are Lipschitz and piecewise \(C^{1}\), the function 
\(h_b(t,q)\) is bounded and piecewise continuously differentiable in \(q\) for each fixed \(t\).
For almost every \(t\), the partial derivative \(\partial_q h_b(t,q)\) exists and satisfies 
\(|\partial_q h_b(t,q)| \le b\), so that it is dominated by the integrable function 
\(\Phi(t) \equiv b\) on \([0,1]\), independently of \(q\).
To justify differentiating under the integral sign, consider for \(h\neq 0\) the difference quotient
\[
f_h(t) = \frac{h_b(t,q+h) - h_b(t,q)}{h}.
\]
For each fixed \(q\), the pointwise limit \(\lim_{h\to 0} f_h(t)\) exists for almost every \(t\) and equals \(\partial_q h_b(t,q)\).
Moreover, \(|f_h(t)| \le b\) for all \(t\) and small \(h\).
Hence, by the dominated convergence theorem, 
\[
\lim_{h\to 0} \int_0^1 f_h(t)\de t
= \int_0^1 \lim_{h\to 0} f_h(t)\de t
= \int_0^1 \partial_q h_b(t,q)\de t.
\]
Therefore,
\[
v'(q)
= \lim_{h\to 0} \frac{v(q+h)-v(q)}{h}
= \int_0^1 \partial_q h_b(t,q)\de t,
\]
and the derivative depends continuously on \(q\).
Since \(h_b(t,q)\) decreases in \(q\) for every \(t\), it follows that \(v'(q)<0\) for all \(q\in(-1/b,1)\).

As observed above, since \(v(-1/b)=1\) and \(v(1)=0\), the range of \(q\) is \([{-}1/b,1]\).
Substituting the section formulas into \eqref{eq:xi} and \eqref{eq:intro-nu-forms} and then changing variables $v\mapsto q$
gives the desired one–dimensional representations.
Indeed, write $q=q(v)$ for the (strictly decreasing) inverse of $v=v(q)$, so that $-\de v = v'(q)\de q$ and $v\in[0,1]$ corresponds to $q\in[-1/b,1]$.

\smallskip
\emph{For $\xi$:}
\[
\xi(C_b)
= 6\int_0^1\Big(a(q(v)) + b^2\big(F(X_a(q(v));q(v)) - F(X_s(q(v));q(v))\big)\Big)\de v - 2,
\]
by \eqref{eq:sq-1D} and \eqref{eq:both_measures}.
Set
\(
G(q):=a(q) + b^2\!\big(F(X_a(q);q) - F(X_s(q);q)\big).
\)
Using $v=v(q)$ and $-\de v=v'(q)\de q$, one has
\[
\int_0^1 G(q(v))\de v
= \int_{-1/b}^{1} G(q)\,(-v'(q))\de q,
\]
whence
\[
\xi(C_b)
= 6\!\!\int_{-1/b}^{1}\!\!\Big(a(q) + b^2\big(F(X_a(q);q) - F(X_s(q);q)\big)\Big)\,(-v'(q))\de q - 2.
\]

\smallskip
\emph{For $\nu$:}
\begin{align*}
\nu(C_b)
= 12\int_0^1\!\Big(\tfrac{1-(1-a(q(v)))^3}{3}
   + b\big(S(X_a(q(v));q(v)) - S(X_s(q(v));q(v))\big)\Big)\de v - 2
\end{align*}
by \eqref{eq:wgt-1D} and \eqref{eq:both_measures}. Set
\(
H(q):=\tfrac{1-(1-a(q))^3}{3} + b\big(S(X_a(q);q) - S(X_s(q);q)\big).
\)
The same substitution yields
\[
\int_0^1 H(q(v))\de v
= \int_{-1/b}^{1} H(q)\,(-v'(q))\de q,
\]
and therefore
\[
\nu(C_b)
= 12\!\!\int_{-1/b}^{1}\!\!\Big(\tfrac{1-(1-a(q))^3}{3}
   + b\big(S(X_a(q);q) - S(X_s(q);q)\big)\Big)\,(-v'(q))\de q - 2.
\]
These are exactly the asserted formulas \eqref{eq:xi-1D} and \eqref{eq:nu-1D}.
\end{proof}

For fixed \(b>0\), each section \(t\mapsto h_b(t,v)\) is a clamped decreasing function of the form
\begin{equation}\label{eq:section_appendix}
h_b(t,v)=\clamp\!\Big(b\big((1-t)^2-q(v)\big),0,1\Big),
\end{equation}
and the (unique) section parameter \(q(v)\) is determined by the marginal constraint
\[
\int_0^1 h_b(t,v)\de t \;=\; v.
\]
Geometrically, \(q\) fixes where the section equals its upper clamp \(1\) (the ``plateau''), where it follows the parabola, and where it is clamped at \(0\).
It is convenient to write
\begin{equation}\label{eq:RrDelta}
\Delta:=R-r\ge 0,
\end{equation}
where \(R\) and \(r\) represent the unclamped switching points for the upper and lower clamps, respectively, as defined in \eqref{eq:unclamped_switches}.

Depending on the value of \(q\), the parabola \(t \mapsto b((1-t)^2 - q)\) from \eqref{eq:section_appendix} meets the clamping bounds \(0\) and~\(1\) in different ways, giving rise to four distinct regimes.
For \(q<0\), the parabola lies entirely above the \(t\)-axis; if it reaches the upper clamp at~1, a flat plateau appears (\emph{upper–clamped regime}), while for weaker curvature (\(b\le1\)) the curve stays below~1 and no clamping occurs (\emph{unclamped regime}).
For small positive~\(q\), the parabola intersects both clamps, yielding a truncated section bounded by two switching points \(a(v)\) and \(s(v)\) (\emph{double–clamped regime}, possible only when \(b>1\)).
For large~\(q\), the parabola lies below the axis and touches only the lower clamp at~0 (\emph{lower–clamped regime}).
These four configurations correspond exactly to the analytic cases in Lemma~\ref{lem:vq} and all occur in Figure~\ref{fig:cvxpy_xi_nu} (e.g., in the left plot there, the upper-clamped case occurs for $v=0.9$, the unclamped case for $v=0.5$, the lower-clamped case for $v=0.1$; the double-clamped case occurs in the right plot for $v=0.5$).

\begin{lemma}[Integration substitution formulas for $v(q)$]\label{lem:vq}
The successive changes of variables $v=v(q)$ and $q=q(R)$ (or $q=q(r)$)
used in the integrals \eqref{eq:xi-1D} and \eqref{eq:nu-1D} of Lemma~\ref{lem:1D} are given by the following
explicit formulas:
\begin{enumerate}[label=(\roman*)]
\item Upper--clamped regime (\(q<0\) and \(R<1\)): if \(q=R^2-\tfrac{1}{b}\), then
\(-\frac{dq}{dR}v'(q)=2bR^2.\)

\item Unclamped regime (\(q<0\) and \(R\ge1\); only if \(b\le1\)): if \(q=R^2-\tfrac{1}{b}\), then
\(-\frac{dq}{dR}v'(q)=2bR.\)

\item Double--clamped regime (\(0\le q<1-\tfrac{1}{b}\); only if \(b>1\)): if \(q=R^2-\tfrac{1}{b}\) and \(r=\sqrt{R^2-\tfrac{1}{b}}\), then
\(-\frac{dq}{dR}v'(q)=1+b(R-r)^2.\)

\item Lower--clamped regime (\(\max\{1-\tfrac{1}{b},0\}\le q\le 1\)): if \(q=r^2\), then
\(-\frac{dq}{dr}v'(q)=2br(1-r).\)
\end{enumerate}
\end{lemma}

\begin{proof}[Proof of Lemma~\ref{lem:vq}]
From \eqref{eq:marg-1D} in Lemma~\ref{lem:1D}, the marginal reads
\[
v(q)=a(q)+b\big(T(X_a(q);q)-T(X_s(q);q)\big),
\qquad 
T(x;q)=\frac{x^3}{3}-qx.
\]
We treat the four regimes separately, using the notation from \eqref{eq:unclamped_switches}, \eqref{eq:section_switches} and \eqref{eq:RrDelta}.

\smallskip\noindent
\emph{(i) Upper–clamped regime (\(q<0\) and \(R<1\)).} 
Here \(a=1-R\), \(X_a=R\), \(X_s=0\), and \(q=R^2-1/b\).
Then
\[
v(q)=1-R+b\Big(\frac{R^3}{3}-qR\Big)
=1-R+b\Big(\frac{R^3}{3}-(R^2-1/b)R\Big)
=1-\frac{2b}{3}R^3.
\]
Since \(R=\sqrt{q+1/b}\), it is
\(
v'(q)=\frac{\de v}{\de R}\frac{\de R}{\de q}
=\Big(-2bR^2\Big)\frac{1}{2R}
=-\,bR.
\)
Moreover \(dq/dR=2R\). Hence
\(
-\frac{dq}{dR}v'(q)=-(2R)(-bR)=2bR^2.
\)

\smallskip\noindent
\emph{(ii) Unclamped regime (\(q<0\) and \(R\ge1\); only if \(b\le1\)).} 
Here \(a=0\), \(X_a=1\), \(X_s=0\).
Then
\[
v(q)=b\big(T(1;q)-T(0;q)\big)
=b\Big(\frac{1}{3}-q\Big),
\qquad
v'(q)=-\,b.
\]
Since again \(q=R^2-1/b\), we have \(dq/dR=2R\), and therefore
\(
-\frac{dq}{dR}v'(q)=-(2R)(-b)=2bR.
\)

\smallskip\noindent
\emph{(iii) Double–clamped regime (\(0\le q<1-\tfrac{1}{b}\); only if \(b>1\)).}
Here \(a=1-R\), \(X_a=R\), \(X_s=r\).
Using \(R^2=r^2+1/b\) and \(q=r^2\),
\[
v(q)=1-R+b\Big(\frac{R^3-r^3}{3}-q(R-r)\Big).
\]
Differentiating with respect to \(q\), noting \(R'=\frac{1}{2R}\) and \(r'=\frac{1}{2r}\),
\begin{align*}
v'(q)
=-\,\frac{1}{2R}
 + b\Big(\underbrace{(R^2-q)R' - (r^2-q)r'}_{=\;(1/b)\,R'} - (R-r) - q(R'-r')\Big)
= -\,\frac{1}{2R}\Big(1+b(R-r)^2\Big).
\end{align*}
Since \(dq/dR=2R\), we obtain
\(
-\frac{dq}{dR}v'(q)=-(2R)\Big(-\frac{1}{2R}(1+b(R-r)^2)\Big)
=1+b(R-r)^2.
\)

\smallskip\noindent
\emph{(iv) Lower–clamped regime (\(\max\{1-\tfrac{1}{b},0\}\le q\le 1\)).} 
Here \(a=0\), \(X_a=1\), \(X_s=r\).
Then
\[
v(q)=b\big(T(1;q)-T(r;q)\big)
=b\Big(\frac{1}{3}-\frac{r^3}{3}+q\,r-q\Big)
=b\Big(\frac{1}{3}-r^2+\frac{2}{3}r^3\Big),
\]
since \(q=r^2\).
Differentiating gives
\(
v'(q)=b\Big(-2r\frac{1}{2r}+2r^2\frac{1}{2r}\Big)
=b(r-1).
\)
Since \(dq/dr=2r\), it follows that
\(
-\frac{dq}{dr}v'(q)=-(2r)b(r-1)=2br(1-r).
\)

\smallskip
The claimed continuity at the regime boundaries follows by direct substitution and \(v'(q)<0\) in each regime.
\end{proof}
We now combine the one-dimensional representations from Lemma~\ref{lem:1D} with the substitution formulas from Lemma~\ref{lem:vq} to derive the closed forms for \(\xi(C_b)\) and \(\nu(C_b)\).

\begin{proof}[Proof of Theorem~\ref{thm:closed}]
We start from the one-dimensional integral representations in Lemma~\ref{lem:1D}:
\begin{align*}
\xi(C_b)=~&6\!\!\int_{-1/b}^{1}\!\!\Big(a(q)+b^2\big(F(X_a;q)-F(X_s;q)\big)\Big)\,(-v'(q))\de q-2, \\
\nu(C_b)=~&12\!\!\int_{-1/b}^{1}\!\!\Big(\frac{1-(1-a)^3}{3}+b\big(S(X_a;q)-S(X_s;q)\big)\Big)\,(-v'(q))\de q-2,
\end{align*}
with $F(x;q)=\frac{x^5}{5}-\frac{2q}{3}x^3+q^2x$ and $S(x;q)=\frac{x^5}{5}-\frac{q}{3}x^3$.
We evaluate these integrals by splitting into the cases $b\le 1$ and $b>1$, using the integral substitutions for $(-v'(q))\de q$ from Lemma~\ref{lem:vq} and the corresponding section parameters $(a, X_a, X_s)$ for each regime.

\smallskip
\noindent\emph{Case \(0<b\le 1\).}
Using Lemma~\ref{lem:1D} and Lemma~\ref{lem:vq}, the \(q\)–integral splits into the following regimes:
\begin{enumerate}[label=(\roman*)]
\item \emph{Upper–clamped:} 
\(
q\in\bigl(-\tfrac{1}{b},\,1-\tfrac{1}{b}\bigr), 
\quad 
R\in(0,1),
\quad 
 -v'(q)\de q = 2b\,R^{2}\de R.
\)

\item \emph{Unclamped:} 
\(
q\in\bigl[\,1-\tfrac{1}{b},\,0\,\bigr),
\quad 
R\in\bigl[1,\,\tfrac{1}{\sqrt b}\bigr),
\quad 
 -v'(q)\de q = 2b\,R\de R.
\)

\item[(iv)] \emph{Lower–clamped:} 
\(
q\in[0,1],
\quad 
r:=\sqrt q\in(0,1],
\quad 
 -v'(q)\de q = 2b\,r(1-r)\de r.
\)
\end{enumerate}

\noindent
In each regime, the corresponding section parameters are
\begin{enumerate}[label=(\roman*)]
\item \(a = 1-R,\quad X_a = R,\quad X_s = 0,\quad q = R^{2}-\tfrac{1}{b}.\)

\item \(a = 0,\quad X_a = 1,\quad X_s = 0,\quad q = R^{2}-\tfrac{1}{b}.\)

\item[(iv)] \(a = 0,\quad X_a = 1,\quad X_s = r,\quad q = r^{2}.\)
\end{enumerate}
and with \(F(x;q)=\frac{x^5}{5}-\frac{2q}{3}x^3+q^2x\), \(S(x;q)=\frac{x^5}{5}-\frac{q}{3}x^3\) from \eqref{eq:primitives}, Lemma~\ref{lem:1D} gives the regime integrands
\[
\begin{aligned}
G &:= a + b^2\!\left(F(X_a;q)-F(X_s;q)\right),
&\qquad
H &:= \frac{1-(1-a)^3}{3} + b\!\left(S(X_a;q)-S(X_s;q)\right),
\end{aligned}
\]
for \(\xi\) and \(\nu\), respectively. Hence
\[
\begin{aligned}
\xi(C_b)
&= 6\Bigg(
   2b\!\int_{0}^{1}\! G_i(R)\,R^2\de R
 + 2b\!\int_{1}^{1/\sqrt b}\! G_{ii}(R)\,R\de R
 + 2b\!\int_{0}^{1}\! G_{iv}(r)\,r(1-r)\de r
\Bigg) - 2,\\[4pt]
\nu(C_b)
&= 12\Bigg(
   2b\!\int_{0}^{1}\! H_i(R)\,R^2\de R
 + 2b\!\int_{1}^{1/\sqrt b}\! H_{ii}(R)\,R\de R
 + 2b\!\int_{0}^{1}\! H_{iv}(r)\,r(1-r)\de r
\Bigg) - 2,
\end{aligned}
\]
where, explicitly,
\[
\begin{aligned}
G_i(R)&=(1-R)+b^2\!\left(\frac{R^5}{5}-\frac{2q}{3}R^3+q^2R\right)\!,
& H_i(R)&=\frac{1-R^3}{3}+b\!\left(\frac{R^5}{5}-\frac{q}{3}R^3\right)\!,
& q&=R^2-\tfrac{1}{b},\\[3pt]
G_{ii}(R)&=b^2\!\left(\frac{1}{5}-\frac{2q}{3}+q^2\right)\!,
& H_{ii}(R)&=b\!\left(\frac{1}{5}-\frac{q}{3}\right)\!,
& q&=R^2-\tfrac{1}{b},\\[3pt]
G_{iv}(r)&=b^2\!\left(\frac{1}{5}-\frac{2r^2}{3}+r^4 - \frac{r^5}{5}\right)\!,
& H_{iv}(r)&=b\!\left(\frac{1}{5}-\frac{r^2}{3} - \frac{r^5}{5} + \frac{r^5}{3}\right)\!.
\end{aligned}
\]
All six regime integrals are elementary polynomials in \(R\) (or \(r\)) and \(b\), and can be evaluated symbolically. The three contributions to \(\xi(C_b)\) read
\[
\begin{aligned}
I^{(i)}_{\xi,A}
&:=\int_{0}^{1} G_i(R)\,\bigl(2bR^2\bigr)\de R
   \;=\; \frac{2b\bigl(3b^2-10b+15\bigr)}{45},\\[2mm]
I^{(ii)}_{\xi,A}
&:=\int_{1}^{1/\sqrt b} G_{ii}(R)\,\bigl(2bR\bigr)\de R
   \;=\; -\frac{b^3}{5}+\frac{8b^2}{15}-\frac{2b}{3}+\frac13,\\[2mm]
I^{(iv)}_{\xi,A}
&:=\int_{0}^{1} G_{iv}(r)\,\bigl(2br(1-r)\bigr)\de r
   \;=\; \frac{b^3}{35}.
\end{aligned}
\]
Hence
\[
\xi(C_b)
\;=\;
6\Bigl(I^{(i)}_{\xi,A}+I^{(ii)}_{\xi,A}+I^{(iv)}_{\xi,A}\Bigr)-2
\;=\;
\frac{8b^{2}(7-3b)}{105}.
\]

Similarly, for \(\nu(C_b)\) one obtains
\[
\begin{aligned}
I^{(i)}_{\nu,A}
&:=\int_{0}^{1} H_i(R)\,\bigl(2bR^2\bigr)\de R
   \;=\;\frac{b(20-3b)}{90},\\[2mm]
I^{(ii)}_{\nu,A}
&:=\int_{1}^{1/\sqrt b} H_{ii}(R)\,\bigl(2bR\bigr)\de R
   \;=\;-\frac{b^2}{30}-\frac{2b}{15}+\frac16,\\[2mm]
I^{(iv)}_{\nu,A}
&:=\int_{0}^{1} H_{iv}(r)\,\bigl(2br(1-r)\bigr)\de r
   \;=\;\frac{4b^2}{105}.
\end{aligned}
\]
Thus
\(
\nu(C_b)
=
12\Bigl(I^{(i)}_{\nu,A}+I^{(ii)}_{\nu,A}+I^{(iv)}_{\nu,A}\Bigr)-2
=
\frac{4b(28-9b)}{105},
\)
which is the claimed formula for \(\nu(C_b)=\NN(b)\) in the case \(0<b\le 1\).

\smallskip
\noindent\emph{Case \(b>1\).}
Here the \(q\)–integral splits into regimes (i), (iii), and (iv):
\begin{enumerate}[label=(\roman*)]
\item \emph{Upper–clamped:}
\(
R\in\Bigl(0,\tfrac{1}{\sqrt b}\Bigr),\quad
  -v'(q)\de q=2b\,R^{2}\de R.
\)
\item[(iii)] \emph{Double–clamped:}
\(
R\in\Bigl(\tfrac{1}{\sqrt b},1\Bigr),\quad r=\sqrt{R^2-\tfrac{1}{b}},\quad
  (-v'(q))\de q=\bigl(1+b(R-r)^2\bigr)\de R.
\)
\item[(iv)] \emph{Lower–clamped:}
\(
r\in\Bigl(\sqrt{1-\tfrac{1}{b}},\,1\Bigr],\quad
  (-v'(q))\de q=2b\,r(1-r)\de r.
\)
\end{enumerate}
\noindent
In each regime, the corresponding section parameters are
\begin{enumerate}[label=(\roman*)]
\item \(a = 1-R,\quad X_a = R,\quad X_s = 0,\quad q = R^{2}-\tfrac{1}{b}.\)
\item[(iii)] \(a = 1-R,\quad X_a = R,\quad X_s = r,\quad q = R^{2}-\tfrac{1}{b}.\)
\item[(iv)] \(a = 0,\quad X_a = 1,\quad X_s = r,\quad q = r^{2}.\)
\end{enumerate}
Thus,
\small
\[
\begin{aligned}
\xi(C_b)
&= 6\Bigg(
   2b\int_{0}^{\frac1{\sqrt b}} G_i(R)R^2\de R
 + \int_{\frac1{\sqrt b}}^{1} G_{iii}(R)\bigl(1+b(R-r)^2\bigr)\de R
 + 2b\int_{\sqrt{1-\frac1b}}^{1} G_{iv}(r)r(1-r)\de r
\Bigg) - 2,\\[4pt]
\nu(C_b)
&= 12\Bigg(
   2b\int_{0}^{\frac1{\sqrt b}} H_i(R)\,R^2\de R
 + \int_{\frac1{\sqrt b}}^{1} H_{iii}(R)\,\bigl(1+b(R-r)^2\bigr)\de R
 + 2b \int_{\sqrt{1-\frac1b}}^{1} H_{iv}(r)\,r(1-r)\de r
\Bigg) - 2,
\end{aligned}
\]
\normalsize
with $G_i$, $H_i$, $G_{iv}$, and $H_{iv}$ as before, and
\[
\begin{aligned}
G_{iii}(R)&=(1-R)+b^2\!\left(\frac{R^5}{5}-\frac{2q}{3}R^3+q^2R - \frac{r^5}{5}+\frac{2q}{3}r^3-q^2r\right),\\
H_{iii}(R)&=\frac{1-R^3}{3}+b\!\left(\frac{R^5}{5}-\frac{q}{3}R^3 - \frac{r^5}{5}+\frac{q}{3}r^3\right),
\qquad q=R^2-\tfrac{1}{b},\;\; r=\sqrt{R^2-\tfrac{1}{b}}
\end{aligned}
.\]

\smallskip
\noindent\emph{Integrals for \(\xi(C_b)\):}
\begin{align*}
I^{(i)}_{\xi,B}
&=\int_{0}^{1/\sqrt b} G_i(R)\,\bigl(2bR^2\bigr)\de R
   \;=\; -\frac{14}{45b}+\frac{2}{3\sqrt b},\\[2mm]
I^{(iii)}_{\xi,B}
&=\int_{1/\sqrt b}^{1} G_{iii}(R)\,\bigl(1+b(R-r)^2\bigr)\de R\\
&=\;\frac{1}{360b^{3/2}}\!\Bigl(
96b^{9/2}-352b^{7/2}+528b^{5/2}-192b^{3/2}
+15\sqrt b\,\log b -30\sqrt b\,\log\!\bigl(\sqrt b\,\sqrt{b-1}+b\bigr)\\
&
+160\sqrt b -96b^4\sqrt{b-1}
+304b^3\sqrt{b-1} -388b^2\sqrt{b-1}
+210b^{3/2}\sqrt{b-1}
+10b\sqrt{b-1}
-240b
\Bigr),\allowdisplaybreaks\\[2mm]
I^{(iv)}_{\xi,B}
&=\int_{\sqrt{1-1/b}}^{1} G_{iv}(r)\,\bigl(2br(1-r)\bigr)\de r\\
&=\;\frac{1}{105b}\!\Bigl(
-32b^{9/2}+112b^{7/2}-154b^{5/2}+91b^{3/2}
-14\sqrt b+32b^4\sqrt{b-1}-96b^3\sqrt{b-1}\\
&\hspace{2cm}
+110b^2\sqrt{b-1}-46b\sqrt{b-1}
\Bigr).
\end{align*}

\smallskip
\noindent\emph{Integrals for \(\nu(C_b)\):}
\begin{align*}
I^{(i)}_{\nu,B}
&=\int_{0}^{1/\sqrt b} H_i(R)\,\bigl(2bR^2\bigr)\de R
   \;=\; -\frac{1}{30b}+\frac{2}{9\sqrt b},\\[2mm]
I^{(iii)}_{\nu,B}
&=\int_{1/\sqrt b}^{1} H_{iii}(R)\,\bigl(1+b(R-r)^2\bigr)\de R\\
&=\;\frac{1}{48b^{3/2}}\!\Bigl(
\frac{b^{3/2}}{15}\sqrt{b-1}
-\frac{29}{90}\sqrt b\,\sqrt{b-1}
-\frac{b^2}{15}
+\frac{16b}{45}
-\frac15
+\frac{2}{15b}
+\frac{\log b}{96b}
-\frac{\log\!\bigl(\sqrt b\,\sqrt{b-1}+b\bigr)}{48b}\\
&\hspace{2cm}
+\frac{107\sqrt{b-1}}{360\sqrt b}
-\frac{2b\sqrt{b-1}}{9\sqrt b}
-\frac{\sqrt{b-1}}{3b^{3/2}}
\Bigr),\allowdisplaybreaks\\[2mm]
I^{(iv)}_{\nu,B}
&=\int_{\sqrt{1-1/b}}^{1} H_{iv}(r)\,\bigl(2br(1-r)\bigr)\de r\\
&=\;\frac{1}{105b}\!\Bigl(
2\sqrt b\,(b-1)^{3/2}
-\frac{17b^2}{105}
+\frac{b}{5}
+\frac{(b-1)^2}{6}
+\frac{(b-1)^2}{15}
-\frac{2(b-1)^{5/2}}{15\sqrt b}
-\frac{4(b-1)^{7/2}}{105\sqrt b}
\Bigr).
\end{align*}

\smallskip
\noindent\emph{Summing the regimes.}
Adding all regime contributions gives
\[
\xi(C_b)
= 6\Bigl(I^{(i)}_{\xi,B}+I^{(iii)}_{\xi,B}+I^{(iv)}_{\xi,B}\Bigr)-2,
\qquad
\nu(C_b)
= 12\Bigl(I^{(i)}_{\nu,B}+I^{(iii)}_{\nu,B}+I^{(iv)}_{\nu,B}\Bigr)-2.
\]
The logarithmic terms
\(
\frac{\log b}{96b}
-\frac{\log\!\bigl(\sqrt b\,\sqrt{b-1}+b\bigr)}{48b}
\)
combine into a single expression involving  \(\acosh(\sqrt b)\).
Indeed, $\acosh$ is commonly defined by
\(
\acosh(x)=\log\!\bigl(x+\sqrt{x^2-1}\,\bigr),
\)
and satisfies the identity
\(
\acosh(\sqrt b)
=\log\!\bigl(b+\sqrt b\,\sqrt{b-1}\,\bigr)-\tfrac12\log b.
\)
This implies
\[
\log b-2\log\!\bigl(\sqrt b\,\sqrt{b-1}+b\bigr)
=-2\,\acosh(\sqrt b),
\]
so that the above combination of logarithms is proportional to \(\acosh(\sqrt b)\).
Introducing
\(
\gamma:=\sqrt{\frac{b-1}{\,b\,}}
\)
and
\(
A:=\acosh(\sqrt b),
\)
we obtain after simplification the closed forms
\[
\xi(C_b)
= \frac{
 183\,\gamma
 -38\,b\,\gamma
 -88\,b^{2}\gamma
 +112\,b^{2}
 +48\,b^{3}\gamma
 -48\,b^{3}
 -\frac{105\,A}{b}
}{210}
\]
and
\[
\nu(C_b)
= \frac{
 \frac{87\,\gamma}{b}
 +250\,\gamma
 -376\,b\,\gamma
 +448\,b
 +144\,b^{2}\gamma
 -144\,b^{2}
 -\frac{105\,A}{b^{2}}
}{420},
\]
which are the claimed formulas for \(\xi(C_b)=\Xi(b)\) and \(\nu(C_b)=\NN(b)\) in the case \(b>1\).

\smallskip
\noindent\emph{Continuity.}
The two cases agree at \(b=1\), yielding \(\xi(C_1)=32/105\) and \(\nu(C_1)=76/105\).
\end{proof}

\subsection{Proof of Theorem~\ref{thm:KKT}}
\label{sec:proofs-kkt}

We next turn to the variational characterization of the boundary family and verify that the clamped solution satisfies the KKT system.

\begin{proof}[Proof of Theorem~\ref{thm:KKT}]
\emph{Step 1: Lagrangian and KKT system.}
We cast the problem in minimization form and include the box constraints in the Lagrangian.
Let
\[
  f(h):=-\!\!\int_0^1\!\!\int_0^1 2(1-t)^2\,h(t,v)\de t\de v.
\]
The normal-cone condition in Lemma~\ref{lem:KKT-framework}\,\ref{itm:first-order-optimality} can be represented by multipliers
\[
  \mu\ge 0,\qquad \gamma\in L^2(0,1),\qquad \alpha,\beta\in L^2_+((0,1)^2),
\]
see also \cite{rockel2025exact}.
With simple rescaling of the $\mu$ multiplier, the Lagrangian from Lemma~\ref{lem:KKT-framework} with $a=1$ is
\begin{align}\begin{aligned}\label{eq:Lag-ref}
  \mathcal{L}(h;\mu,\gamma,\alpha,\beta)
  \quad=\quad
  &-\!\!\int_0^1\!\!\int_0^1 2(1-t)^2\,h(t,v)\de t\de v
  + \frac{\mu}{6}\!\left(6\!\!\int_0^1\!\!\int_0^1 h(t,v)^2\de t\de v - 2 - c\right) \\
  &+ \int_0^1 \gamma(v)\!\left(\int_0^1 h(t,v)\de t - v\right)\de v \\
  &+ \int_0^1\!\!\int_0^1 \alpha(t,v)\,(-h(t,v))\de t\de v
  + \int_0^1\!\!\int_0^1 \beta(t,v)\,(h(t,v)-1)\de t\de v.
\end{aligned}\end{align}
For any $k\in L^2((0,1)^2)$, the Fréchet derivative of the Lagrangian $\mathcal{L}(h;\mu,\gamma,\alpha,\beta)$ in \eqref{eq:Lag-ref} with respect to $h$ is
\[
  \mathrm{D}_h\mathcal{L}(h;\mu,\gamma,\alpha,\beta)[k]
  \;=\;
  \int_0^1 \int_0^1\big(-2(1-t)^2+2\mu\,h(t,v)+\gamma(v)-\alpha(t,v)+\beta(t,v)\big)\,k(t,v)\de t\de v,
\]
for all $h,k\in L^2((0,1)^2)$.
The Karush–Kuhn–Tucker (KKT) conditions are then:
\begin{enumerate}[label=(\roman*)]
  \item \emph{Stationarity:}
  \begin{equation}\label{eq:stat-ref}
    -2(1-t)^2+2\mu\,h(t,v)+\gamma(v)-\alpha(t,v)+\beta(t,v)=0
    \quad\text{for a.e.~}(t,v)\in(0,1)^2.
  \end{equation}

  \item \emph{Primal feasibility:} $h$ satisfies the constraints of Optimization Problem~\ref{opt:main}, i.e.
  \[
     \int_0^1 h(t,v)\de t=v \ \text{ for a.e.\ }v\in(0,1),\qquad
     0\le h\le 1\ \text{a.e.},\qquad
     6\int_0^1\int_0^1 h^2(t,v)\de t\de v-2\le c.
  \]

  \item \emph{Dual feasibility:}
  \(
     \mu\ge 0,\quad \alpha\ge 0\ \text{a.e.},\quad \beta\ge 0\ \text{a.e.}
  \)

  \item \emph{Complementarity:}
  \begin{equation}\label{eq:compl-ref}
     \mu\Big(6\int_0^1\int_0^1 h^2(t,v)\de t\de v-2-c\Big)=0,\quad
     \alpha(t,v)\,h(t,v)=0,\quad
     \beta(t,v)\,(1-h(t,v))=0
  \end{equation}
\end{enumerate}

\smallskip
\emph{Step 2: Construction of a KKT point.}
For $\mu>0$, consider the solution candidate $h_\mu:=h_{1/\mu}$ from \eqref{eq:clamped_h}, that is,
\begin{equation}\label{eq:def-candidate}
  h_\mu(t,v)=\clamp\!\Big(\frac{(1-t)^2-q(v)}{\mu},\,0,\,1\Big),
\end{equation}
and let
\begin{equation}\label{eq:def-alphabeta}
    \gamma(v):=2q(v),\qquad
  \alpha(t,v):=\big(2q(v)-2(1-t)^2\big)_+,\qquad
  \beta(t,v):=\big(2(1-t)^2-2q(v)-2\mu\big)_+.
\end{equation}
By construction, $h_\mu$ satisfies the marginal constraint and $0\le h_\mu\le 1$ a.e.; hence $h_\mu$ is feasible.
Clearly, $\alpha,\beta\in L^2_+((0,1)^2)$ and $\gamma\in L^2(0,1)$ since $q$ is bounded.

\smallskip
\emph{Step 3: Verification of the KKT conditions.}
Primal feasibility follows by construction, and dual feasibility is satisfied as $\mu, \alpha, \beta$ are non-negative and $q$ is bounded. 
Regarding complementarity, we first check 
\[
  \mu\Big(6\int_0^1\int_0^1 h_\mu^2(t,v)\de t\de v-2-c\Big) = \mu(\Xi(1/\mu) - c) = 0.
\]
For \(0<b\le 1\), differentiating \eqref{eq:Xi-formula} gives
\[
\Xi'(b)=\frac{8b(14-9b)}{105}>0.
\]
For \(b>1\), it is more convenient to reparametrize by
\[
\gamma:=\sqrt{\frac{b-1}{b}}\in(0,1).
\]
Then
\[
b=\frac{1}{1-\gamma^2},
\qquad
A:=\acosh(\sqrt b)=\frac12\log\!\frac{1+\gamma}{1-\gamma},
\qquad
\frac{db}{d\gamma}=\frac{2\gamma}{(1-\gamma^2)^2}.
\]
Substituting \(b=(1-\gamma^2)^{-1}\) into \eqref{eq:Xi-formula} and applying the chain rule \(\Xi'(b) = \frac{d\Xi}{d\gamma} \big/ \frac{db}{d\gamma}\), a straightforward simplification yields
\[
\Xi'(b)
=
\frac{(1-\gamma)^2}{210(1+\gamma)^2}
\Bigl(
105\gamma^4A
+420\gamma^3A
+39\gamma^3
+630\gamma^2A
+156\gamma^2
+420\gamma A
+215\gamma
+105A
+80
\Bigr).
\]
Since \(\gamma\in(0,1)\) and \(A=\frac12\log\!\frac{1+\gamma}{1-\gamma}>0\), every term inside the parenthesis is strictly positive. Hence,
\[
\Xi'(b)>0
\qquad\text{for all } b>1.
\]
Therefore, \(b\mapsto \Xi(b)\) is strictly increasing on \((0,\infty)\).
Consequently, since \(\mu\mapsto 1/\mu\) is strictly decreasing on \((0,\infty)\), the map
\[
\mu\mapsto \Xi(1/\mu)
=
6\int_0^1\int_0^1 h_\mu^2(t,v)\de t\de v-2
\]
is strictly decreasing on \((0,\infty)\).
Moreover, Theorem~\ref{thm:closed} yields the limits
\begin{equation}\label{eq:limits_xi_mu}
\lim_{\mu\downarrow 0}6\int_0^1\int_0^1 h_\mu^2(t,v)\de t\de v-2=1,\qquad
\lim_{\mu\uparrow\infty}6\int_0^1\int_0^1 h_\mu^2(t,v)\de t\de v-2=0.
\end{equation}
Hence, there exists a unique $\mu_c>0$ such that
\(
\xi(C_{1/\mu_c})=c,
\)
and choosing $\mu=\mu_c$ makes this complementarity hold.
Regarding the complementarity for the box constraints, note that from \eqref{eq:def-candidate}–\eqref{eq:def-alphabeta},
\[
  \{h_\mu=0\}=\{(1-t)^2\le q(v)\},\quad
  \{h_\mu=1\}=\{(1-t)^2\ge q(v)+\mu\}.
\]
On $\{h_\mu=0\}$, it is
\[
\alpha=2(q-(1-t)^2)\ge 0, \quad \beta=0,
\]
hence $\alpha\,h_\mu=0$, $\beta(1-h_\mu)=0$.
Furthermore, on $\{0<h_\mu<1\}$, it is
\(
\alpha=\beta=0,
\)
hence $\alpha\,h_\mu=\beta(1-h_\mu)=0$.
Lastly, on $\{h_\mu=1\}$, it is
\[
\alpha=0,\quad\beta=2\big((1-t)^2-q-\mu\big)\ge 0,
\]
hence again $\alpha\,h_\mu=\beta(1-h_\mu)=0$.
Thus, the box-complementarity in \eqref{eq:compl-ref} holds.

Let now $x:=(1-t)^2$ and $q:=q(v)$ for brevity.
We verify the stationarity condition
\[
  -2x+2\mu\,h_\mu+2q-\alpha+\beta=0 \quad\text{a.e.}
\]
from \eqref{eq:stat-ref} by three cases.
If $x\le q$ (i.e.\ $h_\mu=0$), then $\alpha=2(q-x)$ and $\beta=0$, hence
\[
 -2x+0+2q-(2(q-x))+0=0.
\]
If $q<x<q+\mu$ (i.e.\ $0<h_\mu<1$), then $\alpha=\beta=0$ and $h_\mu=(x-q)/\mu$, hence
\[
 -2x+2\mu\frac{x-q}{\mu}+2q=0.
\]
Lastly, if $x\ge q+\mu$ (i.e.\ $h_\mu=1$), then $\alpha=0$ and $\beta=2(x-q-\mu)$, hence
\[
 -2x+2\mu+2q-0+2(x-q-\mu)=0.
\]
Therefore, \eqref{eq:stat-ref} holds pointwise a.e.

\smallskip
\emph{Step 4: Optimality and uniqueness.}
At $(h_{\mu_c};\mu_c,\gamma,\alpha,\beta)$ the Lagrangian Hessian is
\[
  D^2_{hh}\mathcal{L}(h_{\mu_c})[k,k]=2\mu_c\|k\|_{L^2((0,1)^2)}^2\quad(\mu_c>0),
\]
for all $k\in L^2((0,1)^2)$, so the quadratic–growth sufficiency condition from Lemma~\ref{lem:KKT-framework}~\ref{itm:quadratic-growth} holds, yielding a strict local minimizer.
Finally, setting $b_c:=1/\mu_c$, we obtain the section-wise form
\[
h_{b_c}(t,v)=\clamp\!\big(b_c((1-t)^2-q(v)),0,1\big)
\]
with \(\xi(C_{b_c})=c\).
Since $f$ is linear and the feasible set is convex, this strict local minimizer is the \emph{unique global} minimizer, equivalently the unique maximizer of $\nu$ at level $\xi=c$.
The section-wise form \eqref{eq:def-candidate} is exactly the claimed clamped decreasing parabola.
\end{proof}

\subsection{Proof of Theorem~\ref{thm:region}}
\label{sec:proofs-region}

The proof of Theorem~\ref{thm:region} relies on a few lemmas, which we give first. 
The starting point is Lemma~\ref{lemshuff}, which provides a continuous transformation of a stochastically increasing copula into a stochastically decreasing copula while preserving the \(\partial_1\)-section structure and moving monotonically in concordance order.
It is used to derive that every value in the interior of the exact $(\xi,\nu)$-region is also attained.

\begin{lemma}[Shuffling lemma, \cite{ansari2025exact}]\label{lemshuff}
    Let \(C\) be a bivariate SI copula and let \((U,V)\) be a random vector with distribution function \(C.\) For \(p\in [0,1],\) consider the transformation \(T_p\colon [0,1]\to [0,1]\) defined by
    \begin{align}
        T_p(u) := \begin{cases}
        u, &\text{if } 0\le u \le 1-p,\\
        1-(u-(1-p)), & \text{if } 1-p < u \le 1.
    \end{cases}
    \end{align}
    Denote by \(C_p\) the distribution function of \((T_p(U),V).\) Then \(C_p\) is a copula having the following properties.
    \begin{enumerate}[label=(\roman*)]
        \item \label{lemshuff1} \(C_0 = C = C^\uparrow\) and \(C_1 = C_\downarrow,\)
        \item \label{lemshuff2} \(C_p =_{\partial_1 S} C\) for all \(p\in [0,1],\)
        \item \label{lemshuff3} \(C_p\geq_{co} C_{p'}\) for all \(p\geq p',\)
        \item \label{lemshuff4} \(C_p\) is continuous in \(p\) with respect to uniform convergence.
    \end{enumerate}
\end{lemma}

In this context of the interior of the $(\xi,\nu)$-region, we shall also use the elementary observation that $\xi$ is path-continuous.

\begin{lemma}[Path continuity of $\xi$]\label{lem:path_continuity_xi}
Let $C_0$ and $C_1$ be bivariate copulas and define, for $\lambda\in[0,1]$,
\(
C_\lambda:=(1-\lambda)C_0+\lambda C_1.
\)
Then the map \([0,1]\to\R\), \(\lambda\mapsto \xi(C_\lambda)\) is continuous.
\end{lemma}

\begin{proof}
Since the class of bivariate copulas is convex, $C_\lambda$ is again a copula for every $\lambda\in[0,1]$. By linearity of the weak partial derivative,
\[
\partial_1 C_\lambda
=
(1-\lambda)\partial_1 C_0+\lambda \partial_1 C_1
=
(1-\lambda)h_0+\lambda h_1
=
h_\lambda
\]
almost everywhere on $(0,1)^2$. Hence, using that
\(
\xi(C_\lambda)=6\int_0^1\int_0^1 h_\lambda(t,v)^2\de t\de v-2
=6\|h_\lambda\|_{L^2}^2-2,
\)
expanding the square in the Hilbert space $L^2((0,1)^2)$ yields
\[
\|h_\lambda\|_{L^2}^2
=
\|(1-\lambda)h_0+\lambda h_1\|_{L^2}^2 
=
(1-\lambda)^2\|h_0\|_{L^2}^2
+
\lambda^2\|h_1\|_{L^2}^2
+
2\lambda(1-\lambda)\langle h_0,h_1\rangle_{L^2},
\]
where $\langle\cdot,\cdot\rangle_{L^2}$ denotes the usual inner product on
$L^2((0,1)^2)$.
Continuity in $\lambda$ is then immediate.
\end{proof}

Finally, before turning to the proof of Theorem~\ref{thm:region}, we establish the relationship between the derivatives of $\Xi$ and $\NN$ already indicated in \eqref{eq:key-derivative}. 

\begin{lemma}[Key derivative identity]\label{lem:key-derivative}
For all \(b>0\), it holds that
\[
\NN'(b)=\frac{\Xi'(b)}{b}.
\]
\end{lemma}

\begin{proof}
\emph{Case \(0<b\le 1\).}
By Theorem~\ref{thm:closed},
\(
\Xi(b)=\frac{8}{15}b^2-\frac{8}{35}b^3
\)
and
\(
\NN(b)=\frac{16}{15}b-\frac{12}{35}b^2,
\)
hence
\[
\Xi'(b)=\frac{16}{15}b-\frac{24}{35}b^2,
\qquad
\NN'(b)=\frac{16}{15}-\frac{24}{35}b=\frac{\Xi'(b)}{b}.
\]

\smallskip\noindent
\emph{Case \(b>1\).}
Write
\(
\Xi(b)=\frac{1}{210}\Big(U(b)\,\gamma+V(b)+W(b)\,A\Big),\quad
\NN(b)=\frac{1}{420}\Big(\widetilde U(b)\,\gamma+\widetilde V(b)+\widetilde W(b)\,A\Big),
\)
where
\[
\begin{aligned}
&U(b)=183-38b-88b^2+48b^3, &&V(b)=112b^2-48b^3, &&W(b)=-\frac{105}{b},\\
&\widetilde U(b)=\frac{87}{b}+250-376b+144b^2, &&
\widetilde V(b)=448b-144b^2, &&
\widetilde W(b)=-\frac{105}{b^{2}},
\end{aligned}
\]
and
\(
\gamma=\sqrt{\frac{b-1}{b}},
\quad A=\acosh(\sqrt b),
\quad
\gamma'=\frac{1}{2b^{2}\gamma},
\quad A'=\frac{1}{2b\gamma}.
\)
Differentiating,
\[
\Xi'(b)=\frac{1}{210}\Big(U'\gamma+U\gamma'+V'+W'A+WA'\Big),\quad
\NN'(b)=\frac{1}{420}\Big(\widetilde U'\gamma+\widetilde U\gamma'+\widetilde V'
+\widetilde W'A+\widetilde W A'\Big).
\]
Consider
\(
\mathcal D \;:=\; 420b\Big(\NN'(b)-\tfrac{1}{b}\Xi'(b)\Big).
\)
Then
\[
\mathcal D
= b\big(\widetilde U'\gamma+\widetilde U\gamma'+\widetilde V'
+\widetilde W'A+\widetilde W A'\big)
-2\big(U'\gamma+U\gamma'+V'+W'A+W A'\big).
\]
Grouping coefficients then gives
\[
\begin{aligned}
&P_1:=b\widetilde U-2U=-279+326b-200b^2+48b^3,
\qquad P_2:=b\widetilde U'-2U'=76-24b-\frac{87}{b},\\
&P_3:=b\widetilde V'-2V'=0,\qquad
P_4:=b\widetilde W'-2W'=0,\qquad
P_5:=b\widetilde W-2W=\frac{105}{b}.
\end{aligned}
\]
Hence
\(
\mathcal D
= P_1\,\gamma' + P_2\,\gamma + P_5\,A'.
\)
Using \(\gamma'=\tfrac{A'}{b}\) and \(\gamma^2=\tfrac{b-1}{b}\),
\[
\mathcal D
= \frac{P_1+105}{2b^{2}\gamma} + P_2\,\gamma
= \frac{1}{\gamma}\,\frac{(P_1+105)+2b(b-1)P_2}{2b^{2}}.
\]
A direct computation gives
\(
P_1+105+2b(b-1)P_2=0,
\)
hence \(\mathcal D=0\), i.e.
\(
\NN'(b)=\Xi'(b)/b
\)
for \(b>1\).
\end{proof}

With the above lemmas, the explicit formulas from Theorem~\ref{thm:closed} and the optimization result from Theorem~\ref{thm:KKT} in hand, we can now complete the proof of the exact attainable region.

\begin{proof}[Proof of Theorem~\ref{thm:region}]
First, choose a bivariate SI copula \(C\) with \(\xi(C)=1\) and \(\nu(C)=1\), for instance \(C=M\).
By the Shuffling Lemma~\ref{lemshuff}, there exists a continuous path \((C_p)_{p\in[0,1]}\) of copulas in the uniform topology from \(C_0=M\) to \(C_1=W\) along which the \(\partial_1\)-section structure is preserved.
In particular, \(\xi(C_p)=1\) for all \(p\in[0,1]\), while \(\nu(C_p)\) varies continuously from \(1\) to \(-1\).
Consequently, the entire vertical segment $\{(1,y): -1\le y\le 1\}$ is attainable.
Next, fix \(b>0\). By Theorem~\ref{thm:closed},
\[
\xi(C_b)=\Xi(b),\qquad \nu(C_b)=\NN(b),
\]
so \(C_b\) attains the upper boundary point \((\Xi(b),\NN(b))\).
By the reflection property \eqref{eq:nu-sigma2} and the invariance of \(\xi\) under \(C\mapsto C^{\sigma_2}\), the reflected copula \(C_{-b}=C_b^{\sigma_2}\) satisfies
\[
\xi(C_{-b})=\Xi(b),\qquad \nu(C_{-b})=-\NN(b),
\]
and therefore attains the lower boundary point \((\Xi(b),-\NN(b))\).
Consider a copula $C^\ast$ that attains the point $(1,\NN(b))$, and for $\alpha\in[0,1]$ define the convex combination
\[
C_\alpha = \alpha C_b + (1-\alpha) C^\ast.
\]
Then, by linearity of \(\nu\) in the copula argument (cf.~\eqref{eq:nu}), it is
\[
\nu(C_\alpha) = \alpha \nu(C_b) + (1-\alpha)\nu(C^\ast) = \alpha \NN(b) + (1-\alpha)\NN(b) = \NN(b).
\]
for all \(\alpha\in[0,1].\)
Further, note that \(\xi\) varies continuously in $\alpha$ by Lemma~\ref{lem:path_continuity_xi}, so that each point between \((\Xi(b),\NN(b))\) and \((1,\NN(b))\) is attained.
Hence, for every \(b>0\), the horizontal segment from \((\Xi(b),\NN(b))\) to \((1,\NN(b))\) is attainable.
By reflection, the same holds for the horizontal segment from \((\Xi(b),-\NN(b))\) to \((1,-\NN(b))\).

We note that both \(\Xi\) and \(\NN\) are strictly increasing on \((0,\infty)\).
For \(\Xi\), this follows from the proof of Theorem~\ref{thm:KKT}: writing \(b=1/\mu\), the map
\(
\mu\mapsto \Xi(1/\mu)
\)
is strictly decreasing on \((0,\infty)\), hence \(b\mapsto \Xi(b)\) is strictly increasing.
For \(\NN\), suppose to the contrary that there exist \(b_1<b_2\) with \(\NN(b_1)\ge \NN(b_2)\).
Consider
\[
C_\lambda:=(1-\lambda)C_{b_1}+\lambda M,
\qquad \lambda\in[0,1].
\]
By convexity of \(\xi\) along such mixtures, there exists \(\lambda_*\in(0,1)\) such that
\(
\xi(C_{\lambda_*})=\Xi(b_2).
\)
By linearity of \(\nu\),
\[
\nu(C_{\lambda_*})
=(1-\lambda_*)\NN(b_1)+\lambda_*\nu(M)
\ge (1-\lambda_*)\NN(b_2)+\lambda_*
> \NN(b_2),
\]
since \(\nu(M)=1\) and \(C_{b_2}\neq M\).
This contradicts Theorem~\ref{thm:KKT}, which identifies \(C_{b_2}\) as the unique maximizer of \(\nu\) at the level \(\xi=\Xi(b_2)\).
Hence \(\NN\) is strictly increasing.

By Theorem~\ref{thm:KKT}, the copula \(C_b\) uniquely maximizes \(\nu\) among all copulas with \(\xi=\Xi(b)\).
Hence the largest attainable \(\nu\)-value at \(x=\Xi(b)\) equals \(\NN(b)\).
By the reflection property \eqref{eq:nu-sigma2}, the smallest attainable \(\nu\)-value equals \(-\NN(b)\).
Together with the horizontal-segment construction above, this shows that the vertical section of \(\RR_{\xi,\nu}\) at \(x=\Xi(b)\) is exactly
\(
[-\NN(b),\NN(b)].
\)
Since $\Xi(b)$ is continuous and strictly increasing on $(0,\infty)$, with
\[
\lim_{b\downarrow 0}\Xi(b)=0,
\qquad
\lim_{b\uparrow\infty}\Xi(b)=1,
\]
it attains every value in $(0,1)$.
Together with the conventions $\Xi(0)=\NN(0)=0$ and $\Xi(\infty)=\NN(\infty)=1$, we conclude that
\begin{align}\label{eq:region-representation}
    \RR_{\xi,\nu}
    =\Bigl\{\,\bigl(\Xi(b),y\bigr):\ -\NN(b)\le y\le \NN(b),\ b\in[0,\infty]\Bigr\}.
\end{align}

\smallskip
\emph{Closedness and convexity.}
By \eqref{eq:region-representation}, together with the endpoint conventions
\(\Xi(0)=\NN(0)=0\) and \(\Xi(\infty)=\NN(\infty)=1\),
and continuity of \(\Xi\) and \(\NN\) on \((0,\infty)\) from Theorem~\ref{thm:closed},
it follows that \(\RR_{\xi,\nu}\) is closed.
To prove convexity, define
\[
\psi(x):=\NN(\Xi^{-1}(x)),\qquad x\in[0,1].
\]
Since \(\Xi\) is continuous and strictly increasing on \((0,\infty)\), this is well defined on \((0,1)\) and extends continuously to \([0,1]\) by \(\psi(0)=0\), \(\psi(1)=1\). Thus,
\[
\RR_{\xi,\nu}=\{(x,y)\in[0,1]\times\R: |y|\le \psi(x)\}.
\]
It therefore suffices to show that \(\psi\) is concave on \([0,1]\). Writing \(x=\Xi(b)\), the chain rule and Lemma~\ref{lem:key-derivative} give
\[
\psi'(x)
=
\frac{\NN'(b)}{\Xi'(b)}
=
\frac{1}{b}.
\]
Since \(b=\Xi^{-1}(x)\) is strictly increasing in \(x\), the function \(x\mapsto 1/b = 1/\Xi^{-1}(x)\) is strictly decreasing on \((0,1)\).
Hence \(\psi'\) is strictly decreasing, and therefore \(\psi\) is strictly concave on \([0,1]\).
Finally, the hypograph of a concave function is convex, and since \(\RR_{\xi,\nu}\) is exactly the symmetric hypograph
\[
\RR_{\xi,\nu}=\{(x,y): -\psi(x)\le y\le \psi(x)\},
\]
the set \(\RR_{\xi,\nu}\) is convex.

\smallskip
\smallskip
\emph{Maximal difference \(\nu-\xi\).}
Set \(\Delta(b):=\NN(b)-\Xi(b)\).
By Lemma~\ref{lem:key-derivative}, for all \(b>0\),
\[
\Delta'(b)
= \NN'(b)-\Xi'(b)
= \frac{\Xi'(b)}{b}-\Xi'(b)
= \Xi'(b)\!\left(\frac{1}{b}-1\right).
\]
The strict monotonicity of \(\Xi\) follows from $\Xi'>0$, which was already established in the proof of Theorem~\ref{thm:KKT}.
Hence
\[
\Delta'(b)\begin{cases}
>0, & 0<b<1,\\
=0, & b=1,\\
<0, & b>1,
\end{cases}
\]
so \(\Delta\) is strictly increasing on \((0,1]\) and strictly decreasing on \([1,\infty)\).
Therefore, \(\Delta\) attains its unique global maximum at \(b=1\), with value
\[
\Delta(1)
=\NN(1)-\Xi(1)
=\frac{76}{105}-\frac{32}{105}
=\frac{44}{105}.
\qedhere
\]
\end{proof}

\bibliographystyle{plainnat}
\bibliography{LiteratureXiNu}
\end{document}